\documentclass[12pt]{amsart}
\usepackage{fullpage}
\usepackage{xspace, xcolor}
\usepackage{mathtools, enumitem}
\usepackage{cite, thmtools, thm-restate}
\usepackage{mathrsfs, amssymb}
\usepackage[T1]{fontenc}
\usepackage[colorlinks, unicode, pdfencoding=auto, hypertexnames=false]{hyperref}
\usepackage{cleveref}
\usepackage{tikz-cd}

\theoremstyle{plain}
\newtheorem{theorem}{Theorem}[section]
\newtheorem{proposition}[theorem]{Proposition}
\newtheorem{lemma}[theorem]{Lemma}
\newtheorem{corollary}[theorem]{Corollary}

\theoremstyle{definition}
\newtheorem{definition}[theorem]{Definition}

\theoremstyle{remark}
\newtheorem{remark}[theorem]{Remark}
 
\newcommand{\et}{\mathrm{\acute{e}t}}
\newcommand{\ab}{\mathrm{ab}}
\newcommand{\tor}{\mathrm{tor}}
\newcommand{\red}{\mathrm{red}}
\newcommand{\dual}{\vee}

\DeclareMathOperator{\Pic}{\mathbf{Pic}}
\DeclareMathOperator{\pic}{\mathrm{Pic}}
\DeclareMathOperator{\Alb}{\mathbf{Alb}}
\DeclareMathOperator{\Hilb}{\mathbf{Hilb}}
\DeclareMathOperator{\Div}{\mathbf{Div}}
\DeclareMathOperator{\Gr}{\mathbf{Gr}}
\DeclareMathOperator{\NNS}{\mathbf{NS}}
\DeclareMathOperator{\NS}{NS}
\newcommand{\mmu}{\boldsymbol{\mu}}

\DeclareMathOperator{\GL}{\mathbf{GL}}
\DeclareMathOperator{\Mat}{\mathbf{Mat}}
\DeclareMathOperator{\codim}{codim}
\DeclareMathOperator{\Fitt}{Fitt}
\DeclareMathOperator{\sgn}{sgn}
\DeclareMathOperator{\Aut}{Aut}
\DeclareMathOperator{\Hom}{Hom}

\DeclareMathOperator{\Spec}{Spec}

\DeclareMathOperator{\im}{im}
\DeclareMathOperator{\id}{id}
\DeclareMathOperator{\charac}{char}
\DeclareMathOperator{\divisor}{div}

\newcommand{\cat}{\mathbin{\smallfrown}}
\newcommand{\Inc}[2]{\operatorname{Inc}(#1, #2)}

\begin{document}
\title{Computing Picard Schemes}

\author{Hyuk Jun Kweon}
\address{Department of Mathematics, Seoul National University, South Korea}
\email{kweon7182@snu.ac.kr}

\author{Madhavan Venkatesh}
\address{Max Planck Institute for Software Systems, Saarbr\"ucken, Germany}
\email{madhavan@mpi-sws.org}
\begin{abstract}
    We present an algorithm to compute the torsion component $\Pic^\tau X$ of the Picard scheme of a smooth projective variety $X$ over a field $k$. Specifically, we describe $\Pic^\tau X$ as a closed subscheme of a projective space defined by explicit homogeneous polynomials. Furthermore, we compute the group scheme structure on $\Pic^\tau X$. As applications, we provide algorithms to compute various homological invariants. Among these, we compute the abelianization of the geometric \'etale fundamental group $\pi^\et_1(X_{\bar{k}}, x)^{\ab}$. Moreover, we determine the Galois module structure of the first \'etale cohomology groups $H^1_{\et}(X_{\bar{k}}, \mathbb{Z}/n\mathbb{Z})$ without requiring $n$ to be prime to the characteristic of $k$.
\end{abstract}

\maketitle

\section{Introduction}

The Picard scheme $\Pic X$, established by Grothendieck \cite{Grothendieck, Grothendieck2}, is the moduli space parametrizing line bundles on a variety $X$. Just as the Jacobian plays a central role in the arithmetic of curves, its higher-dimensional generalization $\Pic X$ is essential in the study of higher-dimensional varieties. We refer the reader to Kleiman \cite{Kle2} for a comprehensive exposition. Despite its fundamental importance, the classical existence proofs for the Picard scheme are highly non-constructive.

Throughout this paper, we assume that $X$ is a smooth connected projective variety over a field $k$. To ensure computability, we also assume that $k$ is finitely generated over its prime field, either $\mathbb{Q}$ or $\mathbb{F}_p$. This entails no loss of generality. Indeed, $X$ descends to a model $X_0$ over a finitely generated subfield $k_0 \subset k$ generated by the coefficients of the homogeneous polynomials defining $X$. Because the formation of the Picard scheme commutes with base change, the computation can be performed over $k_0$.

Let $\Pic^\tau X \subset \Pic X$ be the torsion component. This projective group scheme parametrizes numerically trivial line bundles, or equivalently, line bundles whose classes are torsion in the N\'eron-Severi group. Our first main result is the effective construction of this moduli space.

\begin{restatable*}{theorem}{PicTau}\label{thm:Pic_tau}
  There exists an explicit algorithm to compute the homogeneous equations defining $\Pic^\tau X$ as a closed subscheme of a projective space.
\end{restatable*}

Since the Picard scheme parametrizes line bundles, the tensor product of line bundles endows it with a group scheme structure. Our second main result is the computation of this structure. By computing a morphism $f: Y \to Z$ between projective schemes, we mean computing the polynomials defining its graph $\Gamma_f \subset Y \times Z$. We apply this to the addition, inverse, and identity morphisms.

\begin{restatable*}{theorem}{PicTauGroup}\label{thm:group_structure}
  There exists an explicit algorithm to compute the group scheme structure on $\Pic^\tau X$, consisting of the addition $\alpha$, the inverse $\iota$, and the identity section $\epsilon$.
\end{restatable*}

Conceptually, $\Pic^\tau X$ represents the dual of the universal integral first homology group of $X$. Consequently, any object derived from a well-behaved integral first homology theory can be recovered from $\Pic^\tau X$. As a major application, we present an algorithm to compute the abelianization $\pi^\et_1(X_{\bar{k}}, x)^\ab$ of the geometric \'etale fundamental group. The most challenging aspect of this group is the $p$-power torsion subgroup in characteristic $p > 0$. 

\begin{restatable*}{theorem}{firstHomology}\label{thm:first_homology}
    There exists an explicit algorithm to compute the structure of the profinite abelian group $\pi^\et_1(X_{\bar{k}}, x)^\ab$.
\end{restatable*}

Moreover, we address the computation of the first \'etale cohomology groups $H_{\et}^1(X_{\bar{k}}, \mathbb{Z}/n\mathbb{Z})$. The difficulty here lies in the case where $n$ is not coprime to the characteristic of $k$. We demonstrate that our method allows for the computation of these groups for arbitrary $n$, including the explicit description of the action of the absolute Galois group of $k$.

\begin{restatable*}{theorem}{firstCohomology}
\label{thm:first_cohomology}
    For any integer $n > 0$, there exists an algorithm to compute the $\Aut(\bar{k}/ k)$-module $H_{\et}^1(X_{\bar{k}}, \mathbb{Z}/n\mathbb{Z})$. Specifically, the algorithm determines the following.
    \begin{enumerate}
        \item The finite abelian group $H_{\et}^1(X_{\bar{k}}, \mathbb{Z}/n\mathbb{Z})$.
        \item A finite extension $L$ of $k$ such that the $\Aut(\bar{k}/k)$-action factors through $\Aut(L/k)$.
        \item The action of the finite group $\Aut(L/k)$ on $H_{\et}^1(X_{\bar{k}}, \mathbb{Z}/n\mathbb{Z})$.
    \end{enumerate}
\end{restatable*}

We emphasize that the scheme structure established by Grothendieck is indispensable for these computations, particularly when $\charac k = p > 0$. Igusa \cite{Igusa} demonstrated that $\Pic X$ can be non-reduced in positive characteristic. Far from being a mere pathology, this non-reducedness encodes essential $p$-torsion data of both $\pi^\et_1(X_{\bar{k}}, x)^{\ab}$ and $H_{\et}^1(X_{\bar{k}}, \mathbb{Z}/N\mathbb{Z})$. Consequently, this cohomological information cannot be derived from the abstract Picard group $\pic X_{\bar{k}}$, nor from Matsusaka's Picard variety $(\Pic^0 X)_\red$ \cite{Matsusaka} together with the geometric N\'eron-Severi group $\NS X_{\bar{k}}$.

In addition to these main applications, we provide algorithms for the Albanese variety $\Alb X$ and the torsion subgroup scheme $(\NNS X)_\tor$. Furthermore, we compute the finite quotients $\pi^\et_1(X_{\bar{k}}, x)^{\ab}/n\pi^\et_1(X_{\bar{k}}, x)^{\ab}$ and the flat cohomology groups $H^1_{\mathrm{fppf}}(X_{\bar{k}}, \mmu_n)$, including their explicit Galois module structures.

A primary obstacle to computing the full Picard scheme $\Pic X$ is determining the geometric N\'eron-Severi group $\NS X_{\bar{k}}$. The current best algorithm, proposed by Poonen, Testa, and van Luijk \cite{PTL}, relies conditionally on the Tate conjecture. Despite this, we believe our methods can compute each component individually. Specifically, for any line bundle $\mathcal{L}$ defined over a finite extension of $k$, it should be feasible to compute the components parametrizing line bundles numerically equivalent to one of the Galois conjugates of $\mathcal{L}$. Nevertheless, we restrict our focus to $\Pic^\tau X$ to avoid excessive complexity. We remark that if $k$ is algebraically closed, $\Pic X$ is simply the disjoint union of copies of $\Pic^\tau X$ indexed by the quotient $\NS X / (\NS X)_\tor$.

Our main technical tool is an explicit and systematic use of Grassmannians. We describe closed subschemes of Grassmannians in Stiefel coordinates and convert the resulting equations into Pl\"ucker coordinates. Moreover, to determine the explicit bounds required for our work, we utilize the Gotzmann Persistence Theorem \cite[Satz]{Got} and related results. Beyond the specific results of this work, we expect that our systematic approach can be extended to prove the computability of numerous other moduli spaces.

The paper is organized as follows. \Cref{sec:grassmannian} reviews Stiefel and Pl\"ucker coordinates on Grassmannians and presents an algorithm to convert equations given in Stiefel coordinates into Pl\"ucker coordinates. \Cref{sec:numerical} establishes the numerical bounds required for our work. \Cref{sec:Div_mH_X} details the computation of the moduli space of effective Cartier divisors numerically equivalent to $mH$, where $H$ is a hyperplane section of $X$. In \Cref{sec:Pic_tau_X}, we construct $\Pic^\tau X$ as a quotient, compute its defining equations, and determine its group structure. Finally, \Cref{sec:applications} discusses applications to \'etale fundamental groups and cohomology.

\section{Notation}

Given a field $k$, let $\bar{k}$ be an algebraic closure of $k$. Given a $k$-algebra $R$, let $Y_R \coloneqq Y \times_{\Spec k} \Spec R$ and $V_R \coloneqq V \otimes_k R$ denote the base changes of a scheme $Y$ and a vector space $V$ over $k$, respectively. We use the superscript $^\dual$ to denote the dual of various objects: the dual vector space, the Pontryagin dual of a locally compact group, the Cartier dual of a finite group scheme, or the dual abelian variety. For a graded module $M$, let $M_t$ denote its graded component of degree $t$.

Let $Z \subset X$ be a closed subscheme of $X$. We denote by $\mathscr{I}_{Z/X}$ the ideal sheaf of $Z$ in $X$. We define the saturated ideal $I_{Z/X}$ and the coordinate ring $S_{Z/X}$ by
\[
    I_{Z/X} = \bigoplus_{t \geq 0} H^0(X, \mathscr{I}_{Z/X}(t)) \quad \text{and} \quad S_{Z/X} = \bigoplus_{t \geq 0} H^0(Z, \mathcal{O}_Z(t)),
\]
respectively. When $X = \mathbb{P}^r$, we omit the subscript $/X$ and simply write $\mathscr{I}_Z$, $I_Z$, and $S_Z$. In the specific case where $Z = \mathbb{P}^r$, we denote the homogeneous coordinate ring by $S = k[x_0, \dots, x_r]$.

For a coherent sheaf $\mathcal{F}$, we denote by $\chi(\mathcal{F})$ its Euler characteristic. For a closed subscheme $Z \hookrightarrow \mathbb{P}^r$, we denote by $P_Z(s) = \chi(\mathcal{O}_Z(s))$ its Hilbert polynomial and by $Q_Z(s) = \chi(\mathscr{I}_Z(s))$ the Hilbert polynomial of the ideal sheaf. These polynomials satisfy the relation
\[
    P_Z(s) + Q_Z(s) = \binom{s+r}{r}.
\]

We denote by $\sim$ the linear equivalence of divisors and by $\equiv$ numerical equivalence. For a global section $f$ of a line bundle, we denote by $\divisor f$ its associated divisor. Similarly, for a rational function $f/g$ in the total quotient ring, we define $\divisor(f/g) \coloneqq \divisor f - \divisor g$.

For a vector space $V$, we denote by $\mathbb{P}(V)$ the projective space of one-dimensional subspaces of $V$. We denote by $\Gr(d, V)$ the Grassmannian of $d$-dimensional subspaces of $V$, and write $\Gr(d, n) \coloneqq \Gr(d, k^n)$. We denote by $\Hilb_P X$ and $\Div_P X$ the Hilbert scheme and the moduli space of effective Cartier divisors on $X$ with Hilbert polynomial $P$, respectively. Given a polynomial $Q$, we define $\Hilb^Q X \coloneqq \Hilb_P X$ and $\Div^Q X \coloneqq \Div_P X$, where $P$ is determined by the relation
\[
    P(s) + Q(s) = \binom{s+r}{r}.
\]
When regarding these schemes as subschemes of $\Gr(d, V)$, for an $R$-point $D$, we denote by $[D] \subset V_R$ the corresponding submodule.

For a divisor $D$, we denote by $\Div_D X$ and $\Pic_D X$ the subschemes parametrizing effective divisors and line bundles numerically equivalent to $D$ and $\mathcal{O}_X(D)$, respectively. We denote by $\Pic X$ the Picard scheme of $X$, by $\Pic^\tau X$ the torsion component parametrizing numerically trivial line bundles, and by $\Pic^0 X$ the identity component parametrizing algebraically trivial line bundles. We denote by $\NS X$ the N\'eron-Severi group of $X$.

Let $\mmu_n$ be the group scheme of $n$-th roots of unity. We denote by $\GL_n$ the general linear group scheme of degree $n$. Let $\Mat_{n \times m}$ denote the scheme of $n \times m$ matrices, and we write $\Mat_n$ for the scheme of $n \times n$ square matrices.

For a non-negative integer $n$, we set $[n] \coloneqq \{0, \dots, n-1\}$. Let $\Inc{d}{n}$ be the set of strictly increasing sequences $\alpha \colon [d] \to [n]$. Let $M = (m_{i,j})$ be an $n \times m$ matrix. Given a sequence of row indices $\alpha \colon [a] \to [n]$ and column indices $\beta \colon [b] \to [m]$, we define the $a \times b$ submatrix $M_{\alpha, \beta}$ by
\[
    M_{\alpha, \beta} = (m_{\alpha(i), \beta(j)})_{i \in [a], \, j \in [b]}.
\]
For a set $A$, let $\id_A$ denote the identity map. We define the submatrices formed by the rows indexed by $\alpha$ and the columns indexed by $\beta$ as $M_{\alpha, \bullet} \coloneqq M_{\alpha, \id_{[m]}}$ and $M_{\bullet, \beta} \coloneqq M_{\id_{[n]}, \beta}$, respectively.

For a morphism $f \colon X \to Y$, we denote by $\Gamma_f \subset X \times Y$ its graph. Let $g \colon Y \to Z$ be another morphism. The graph of the composite morphism $g \circ f$ is obtained scheme-theoretically as
\[
    \Gamma_{g \circ f} = \pi_{X, Z}\left( (\Gamma_f \times Z) \cap (X \times \Gamma_g) \right),
\]
where the intersection is taken in $X \times Y \times Z$, and $\pi_{X, Z} \colon X \times Y \times Z \to X \times Z$ denotes the natural projection.

Throughout the paper, let $X \hookrightarrow \mathbb{P}^r$ be a smooth connected projective variety over a field $k$, and let $H$ denote a hyperplane section of $X$. Unless otherwise specified, $R$ denotes a local $k$-algebra. In \Cref{sec:numerical}, however, we relax this assumption and allow $R$ to be an arbitrary commutative $k$-algebra.

We assume that $k$ is a finitely generated field over the prime field $\mathbb{F}$ which is $\mathbb{Q}$ or $\mathbb{F}_p$. Explicitly, we assume that $k$ is presented as
\[
    k \cong \mathbb{F}(x_0, \dots, x_{a-1})[y_0, \dots, y_{b-1}]/\mathfrak{m},
\]
where $\mathfrak{m}$ is a maximal ideal of the polynomial ring $\mathbb{F}(x_0, \dots, x_{a-1})[y_0, \dots, y_{b-1}]$. Under this assumption, there exist explicit algorithms for computing the primary decomposition \cite{Steel, GTZ} and the radical \cite{KrickLogar, Matsumoto} of an ideal, as well as for the factorization of univariate polynomials over $k$ \cite{Steel, DavenportTrager}.

\section{Preliminaries on Grassmannians}\label{sec:grassmannian}

The goal of this section is to recall Stiefel and Pl\"ucker coordinates and to explain how to rigorously define closed subschemes of the Grassmannian $\Gr(d,n)$ and products of Grassmannians using Stiefel coordinates. Furthermore, we present an algorithm to convert the defining equations from Stiefel coordinates to Pl\"ucker coordinates. The existence of this canonical conversion is the only essential prerequisite for the subsequent sections. Thus, readers willing to accept this result may skip the technical details and proceed directly to \Cref{sec:numerical}. We conclude the section by providing an auxiliary lemma, which will be required later. Unless otherwise stated, $R$ denotes a local $k$-algebra throughout the section.


\subsection{Stiefel and Pl\"ucker Coordinates}
Closed subschemes of $\Gr(d,n)$ can generally be described using Pl\"ucker coordinates via the Pl\"ucker embedding $\Gr(d,n)\hookrightarrow \mathbb P(\bigwedge^d k^n)$. However, this approach has two drawbacks. First, the number of Pl\"ucker coordinates is $\binom{n}{d}$, which is often too large. Second, given Pl\"ucker coordinates, it is typically cumbersome to recover the corresponding rank $d$ submodule. Stiefel coordinates resolve both issues. They reduce the number of variables to $dn$, and the column space $\im S$ of the Stiefel matrix $S$ directly represents the rank $d$ submodule.

\begin{definition}
Let $R$ be a local $k$-algebra and let $M\in \Gr(d,n)(R)$. A \emph{Stiefel matrix} at $M$ is a matrix over $R$
\[ 
S = (s_{i,j})_{i\in[n],\, j\in[d]} = \begin{pmatrix} 
s_{0,0} & s_{0,1} & \cdots & s_{0,d-1} \\
s_{1,0} & s_{1,1} & \cdots & s_{1,d-1} \\
\vdots& \vdots& \ddots & \vdots\\
s_{n-1,0} & s_{n-1,1} & \cdots & s_{n-1,d-1} 
\end{pmatrix}
\]
such that $\im S = M$. The entries of $S$ are called the \emph{Stiefel coordinates}.
\end{definition}

Since $R$ is local, every $M \in \Gr(d,n)(R)$ is free and thus admits a Stiefel matrix. Conversely, a matrix $S \in \Mat_{n \times d}(R)$ represents some $M \in \Gr(d,n)(R)$ if and only if at least one $d \times d$ minor of $S$ is a unit. Two Stiefel matrices $S_0$ and $S_1$ represent the same module if and only if $S_1 = S_0 \cdot B$ for some $B \in \GL_d(R)$. Note that when $d=1$, we have $\Gr(1,n) = \mathbb{P}^{n-1}$, and the Stiefel coordinates coincide with the standard projective coordinates.

As with projective coordinates, if $M \in \Gr(d,n)(R)$ is not specified, we regard the symbols $s_{i,j}$ as formal indeterminates. In this setting, we define the polynomial ring
\[k[s] \coloneqq k\big[s_{i,j} \,\big|\, i \in [n], \, j \in [d]\big].\]
We view the argument of a polynomial $f \in k[s]$ as an $n \times d$ matrix. That is, for an $n \times d$ matrix $A = (a_{i,j})$, the value $f(A)$ is obtained by substituting $s_{i,j} = a_{i,j}$ for all $i \in [n]$ and $j \in [d]$.

Over a general ring $R$, a projective module of rank $d$ need not be free, implying that some $M \in \Gr(d,n)(R)$ cannot be represented by a Stiefel matrix. To avoid the resulting technical difficulties, we follow the approach of \cite[p.~753]{HaSt} and assume that $R$ is local. The following standard lemma justifies the sufficiency of this assumption.

\begin{lemma}\label{lem:subscheme_local_determination}
Let $Y$ be a $k$-scheme, and let $Z_0, Z_1 \subset Y$ be two closed subschemes. Then $Z_0 = Z_1$ if and only if $Z_0(R) = Z_1(R)$ for all local $k$-algebras $R$.
\end{lemma}
\begin{proof}
The forward implication is trivial. For the converse, since the problem is local on $Y$, we may assume $Y = \Spec A$. Let $I_0, I_1 \subset A$ be the ideals defining the closed subschemes $Z_0$ and $Z_1$. Let $\mathfrak{p} \subset A$ be an arbitrary prime ideal. By \cite[Exercise~13.53]{AltmanKleiman2013}, it suffices to show that $(I_0)_{\mathfrak{p}} = (I_1)_{\mathfrak{p}}$.

Consider the local ring $R = (A/I_1)_{\mathfrak{p}}$. The canonical map $\phi: A \to R$ satisfies $\phi(I_1) = 0$, so $\phi \in Z_1(R)$. By the hypothesis $Z_0(R) = Z_1(R)$, we have $\phi \in Z_0(R)$, which implies $\phi(I_0) = 0$. Hence, the image of $I_0$ in $A_{\mathfrak{p}}$ is contained in $(I_1)_{\mathfrak{p}}$, meaning that $(I_0)_{\mathfrak{p}} \subseteq (I_1)_{\mathfrak{p}}$. By symmetry, we obtain $(I_1)_{\mathfrak{p}} \subseteq (I_0)_{\mathfrak{p}}$, and thus $(I_0)_{\mathfrak{p}} = (I_1)_{\mathfrak{p}}$.
\end{proof}

Now, we recall the definition of Pl\"ucker coordinates.

\begin{definition}\label{def:plucker_coordinates}
Let $R$ be a local $k$-algebra, and let $S = (s_{i,j})_{i\in[n],\, j\in[d]}$ be a Stiefel matrix representing an $R$-module $M \in \Gr(d,n)(R)$. The Pl\"ucker coordinate $p_\alpha$ corresponding to a strictly increasing sequence $\alpha \in \Inc{d}{n}$ is defined by
\[
p_\alpha \coloneqq \det S_{\alpha,\bullet} = \det \begin{pmatrix}
s_{\alpha(0), 0} & s_{\alpha(0), 1} & \cdots & s_{\alpha(0), d-1} \\
s_{\alpha(1), 0} & s_{\alpha(1), 1} & \cdots & s_{\alpha(1), d-1} \\
\vdots& \vdots& \ddots & \vdots\\
s_{\alpha(d-1), 0} & s_{\alpha(d-1), 1} & \cdots & s_{\alpha(d-1), d-1}
\end{pmatrix}.
\]
\end{definition}

Pl\"ucker coordinates are unique up to a unit factor, since replacing a Stiefel matrix $S$ with $S \cdot B$ for some $B \in \GL_d(R)$ scales the Pl\"ucker matrix by $\det B$. In the absence of a specific $M \in \Gr(d,n)(R)$, we regard the symbols $p_\alpha$ for $\alpha \in \Inc{d}{n}$ as formal indeterminates. We define the polynomial ring
\[k[p] \coloneqq k[p_\alpha \mid \alpha \in \Inc{d}{n}].\]
The Pl\"ucker coordinates serve as the homogeneous coordinates for the Pl\"ucker embedding $\Gr(d,n) \hookrightarrow \mathbb{P}(\bigwedge^d k^n)$. They are subject to the Pl\"ucker relations, which are explicit quadratic polynomial equations generating the ideal $I_{\Gr(d,n)} \subset k[p]$ \cite[p.~1063, (QR)]{KleimanLaksov}\cite[p.~65]{Harris}.

For convenience, we extend the definition of $p_\alpha$ to arbitrary sequences $\alpha \colon [d] \to [n]$, not necessarily strictly increasing, by the formula
\[ 
p_\alpha = \det S_{\alpha,\bullet}. 
\]
If the sequence $\alpha$ contains repeated terms, the corresponding matrix has identical rows, implying $p_\alpha = 0$. If $\alpha$ consists of distinct terms, let $\beta \in \Inc{d}{n}$ be the unique strictly increasing sequence obtained by reordering $\alpha$. Then there exists a unique permutation $\sigma \colon [d] \to [d]$ such that $\alpha = \beta \circ \sigma$. By the alternating property of the determinant, we have
\[
p_\alpha = \sgn(\sigma) p_\beta,
\]
where $\sgn(\sigma)$ denotes the sign of the permutation $\sigma$. For example, 
\[p_{1,1,2} = 0 \text{ and } p_{0,3,2} = -p_{0,2,3}.\]

\begin{definition}
Given $\alpha \in \Inc{d}{n}$, the affine chart $U_\alpha$ of $\Gr(d,n)$ indexed by $\alpha$ is the open subscheme defined by the condition $p_\alpha \neq 0$.
\end{definition}

Let $M \in U_\alpha(R)$ be an $R$-point. If $S \in \Mat_{n \times d}(R)$ is a Stiefel matrix representing $M$, then the $d \times d$ submatrix $S_{\alpha,\bullet}$ is invertible, meaning that $\det(S_{\alpha,\bullet})$ is a unit in $R$. Since the Stiefel matrix is defined up to the right action of $\GL_d(R)$, we may replace $S$ with the normalized matrix $S \cdot (S_{\alpha,\bullet})^{-1}$. Consequently, on the chart $U_\alpha$, there is a unique Stiefel matrix $S$ representing $M$ such that $S_{\alpha,\bullet} = I_d$.

This normalization is the Grassmannian analogue of the standard affine charts on projective space $\mathbb{P}^N$, where a nonzero homogeneous coordinate is normalized to $1$. The remaining $d(n-d)$ entries of $S$ serve as affine coordinates, establishing the isomorphism $U_\alpha \cong \mathbb{A}^{(n-d)d}_k$.

Since the Pl\"ucker coordinates $p_\alpha$ encode the submodule $M$, one expects to reconstruct $M$ directly from these coordinates. More precisely, we will introduce an explicit matrix $P$ expressed in terms of Pl\"ucker coordinates, whose column space coincides globally with $M$ on $\Gr(d,n)$. For $\beta \in \Inc{d-1}{n}$ and $m \in [n]$, let $\beta\cat m$ denote the sequence of length $d$ obtained by appending $m$ to $\beta$. Formally,
\[
(\beta\cat m)(i) = \begin{cases}
\beta(i) & \text{if } i < d-1, \\
m & \text{if } i = d-1.
\end{cases}
\]

\begin{definition}
The \emph{Pl\"ucker matrix} of $\Gr(d,n)$ is the $n \times \binom{n}{d-1}$ matrix defined by
\[
P = (p_{\beta\cat i})_{i \in [n], \, \beta \in \Inc{d-1}{n}}.
\]
\end{definition}

As an example, the Pl\"ucker matrix for $\Gr(2,4)$ is
\[
P = \begin{pmatrix}
p_{0,0} & p_{1,0} & p_{2,0} & p_{3,0} \\
p_{0,1} & p_{1,1} & p_{2,1} & p_{3,1} \\
p_{0,2} & p_{1,2} & p_{2,2} & p_{3,2} \\
p_{0,3} & p_{1,3} & p_{2,3} & p_{3,3}
\end{pmatrix}
= \begin{pmatrix}
0 & -p_{0,1} & -p_{0,2} & -p_{0,3} \\
p_{0,1} & 0 & -p_{1,2} & -p_{1,3} \\
p_{0,2} & p_{1,2} & 0 & -p_{2,3} \\
p_{0,3} & p_{1,3} & p_{2,3} & 0
\end{pmatrix}.
\]
In standard literature, the term Pl\"ucker matrix is defined only for $\Gr(2,4)$ as the skew-symmetric matrix shown above. Our definition generalizes this classical construction to $\Gr(d,n)$ for arbitrary $d$ and $n$. For another example, the Pl\"ucker matrix for $\Gr(3,4)$ is
\begin{align*}
P &= \begin{pmatrix}
p_{0,1,0} & p_{0,2,0} & p_{0,3,0} & p_{1,2,0} & p_{1,3,0} & p_{2,3,0} \\
p_{0,1,1} & p_{0,2,1} & p_{0,3,1} & p_{1,2,1} & p_{1,3,1} & p_{2,3,1} \\
p_{0,1,2} & p_{0,2,2} & p_{0,3,2} & p_{1,2,2} & p_{1,3,2} & p_{2,3,2} \\
p_{0,1,3} & p_{0,2,3} & p_{0,3,3} & p_{1,2,3} & p_{1,3,3} & p_{2,3,3}
\end{pmatrix} \\
&= \begin{pmatrix}
0 & 0 & 0 & p_{0,1,2} & p_{0,1,3} & p_{0,2,3} \\
0 & -p_{0,1,2} & -p_{0,1,3} & 0 & 0 & p_{1,2,3} \\
p_{0,1,2} & 0 & -p_{0,2,3} & 0 & -p_{1,2,3} & 0 \\
p_{0,1,3} & p_{0,2,3} & 0 & p_{1,2,3} & 0 & 0
\end{pmatrix}.
\end{align*}

In general, every entry of $P$ is either $0$ or equal to $\pm p_\alpha$ for a unique $\alpha \in \Inc d n$.

\begin{lemma} \label{lem:im_P_subset_M}
The Pl\"ucker matrix $P$ representing $M \in \Gr(d,n)(R)$ satisfies $\im P \subset M$.
\end{lemma}
\begin{proof}
Let $S$ be a Stiefel matrix representing $M$. Since Pl\"ucker coordinates are unique up to a unit factor, we may assume that $P$ is defined using this Stiefel matrix. Fix an arbitrary $\beta \in \Inc{d-1}{n}$. For each $j \in [d]$, let $\epsilon(j) \colon [d-1] \to [d]$ denote the sequence $(0, 1, \dots, \widehat{j}, \dots, d-1)$. The $(i, \beta)$-entry of $P$ is
\begin{align*}
p_{\beta\cat i} &= \det S_{\beta\cat i, \bullet} \\
 &= \sum_{j=0}^{d-1} \left((-1)^{d-1+j} \det S_{\beta,\epsilon(j)}\right) \cdot s_{i,j},
\end{align*}
by the cofactor expansion along the $i$-th row. Let $c_j = (-1)^{d-1+j} \det S_{\beta,\epsilon(j)}$, which is independent of the row index $i$. Then the $\beta$-th column of $P$ can be expressed as
\[
 P_{\bullet, \beta} = \sum_{j=0}^{d-1} c_j S_{\bullet, j}.
\]
Consequently, $\im P \subset \im S = M$.
\end{proof}

We will eventually show that $\im P = M$. To this end, we demonstrate that the columns of a suitable $n \times d$ submatrix of $P$ generate $M$. For $\alpha \in \Inc{d}{n}$, let $\alpha[j \mapsto m]$ denote the sequence obtained by replacing the $j$-th entry of $\alpha$ with $m$. Formally,
\[
\alpha[j \mapsto m](i) = \begin{cases}
\alpha(i) & \text{if } i \neq j, \\
m & \text{if } i = j.
\end{cases}
\]

\begin{definition}\label{def:plucker_semisubmatrix}
Given $\alpha \in \Inc{d}{n}$, we define the matrix
\[
P_\alpha = (p_{\alpha[j \mapsto i]})_{i \in [n], \, j \in [d]}.
\]
\end{definition}

Every column of $P_\alpha$ corresponds to a column of $P$ or its negative. More precisely, let $\beta \in \Inc{d-1}{n}$ be the sequence obtained from $\alpha$ by removing the $j$-th entry. Then the $j$-th column of $P_\alpha$ is equal to the $\beta$-th column of $P$ up to a sign.

As an example, if $\alpha = 0, 1, 3$, then for $\Gr(3,4)$, we have
\[
P_{0,1,3} = \begin{pmatrix}
p_{0,1,3} & p_{0,0,3} & p_{0,1,0} \\
p_{1,1,3} & p_{0,1,3} & p_{0,1,1} \\
p_{2,1,3} & p_{0,2,3} & p_{0,1,2} \\
p_{3,1,3} & p_{0,3,3} & p_{0,1,3}
\end{pmatrix}
= \begin{pmatrix}
p_{1,3,0} & -p_{0,3,0} & p_{0,1,0} \\
p_{1,3,1} & -p_{0,3,1} & p_{0,1,1} \\
p_{1,3,2} & -p_{0,3,2} & p_{0,1,2} \\
p_{1,3,3} & -p_{0,3,3} & p_{0,1,3}
\end{pmatrix}.
\]

\begin{lemma} \label{lem:P_alpha_eq_S}
Let $M \in U_\alpha(R)$, and let $S$ be the unique Stiefel matrix representing $M$ such that $S_{\alpha, \bullet} = I_d$. Then we have $P_\alpha = S$.
\end{lemma}
\begin{proof}
By definition, the $(i,j)$-entry of $P_\alpha$ is given by
\[
p_{\alpha[j \mapsto i]} = \det(S_{\alpha[j \mapsto i], \bullet}).
\]
Since $S_{\alpha, \bullet} = I_d$, the matrix $S_{\alpha[j \mapsto i], \bullet}$ is obtained from the identity matrix $I_d$ by replacing its $j$-th row with the $i$-th row of $S$. Hence, $\det(S_{\alpha[j \mapsto i], \bullet}) = s_{i,j}$, implying that $P_\alpha = S$.
\end{proof}

\begin{theorem} \label{thm:im_P_eq_M}
The Pl\"ucker matrix $P$ representing $M \in \Gr(d,n)(R)$ satisfies $\im P = M$.
\end{theorem}
\begin{proof}
Since the open charts $\{U_\alpha\}$ cover $\Gr(d,n)$, this follows immediately from Lemma \ref{lem:im_P_subset_M} and Lemma \ref{lem:P_alpha_eq_S}.
\end{proof}

\begin{remark}
The Pl\"ucker embedding $\Gr(d,n) \hookrightarrow \mathbb{P}(\bigwedge^d k^n)$ induces the Serre twisting sheaf $\mathcal{O}_{\Gr(d,n)}(1)$. Let $\mathcal{B} \subset \mathcal{O}_{\Gr(d,n)}^{\oplus n}$ denote the universal subbundle of rank $d$. \Cref{thm:im_P_eq_M} implies that $\im P = \mathcal{B}(1)$. A similar observation is made in \cite[Lemma~5.5]{Kwe2}.
\end{remark}

\begin{remark}
The Pl\"ucker matrix can also be defined in a basis-free manner. Let $V = R^n$ where $R$ is a local ring. Given $p \in \bigwedge^d V$ and $\omega \in \bigwedge^{d-1} V^\dual$, recall the interior product $\omega \mathbin{\lrcorner} p \in V$ defined by the adjunction
\[
\langle \varphi, \omega \mathbin{\lrcorner} p \rangle = \langle \omega \wedge \varphi, p \rangle
\]
for all $\varphi \in V^\dual$ \cite[Chapter~III, \S11, no.~6]{BourbakiAlg1}. Then the Pl\"ucker matrix represents the linear map
\begin{align*}
P \colon \bigwedge^{d-1} V^\dual &\to V \\
\omega &\mapsto \omega \mathbin{\lrcorner} p.
\end{align*}
In this context, \Cref{thm:im_P_eq_M} can be rephrased as follows: if $p = m_0 \wedge \dots \wedge m_{d-1}$ is decomposable, then $\im P$ is exactly the submodule $M$ generated by $m_0, \dots, m_{d-1}$. Furthermore, the Pl\"ucker relations are merely $(\omega \mathbin{\lrcorner} p) \wedge p = 0$ for all $\omega \in \bigwedge^{d-1} V^\dual$. However, since our primary focus is on explicit computation, we rely on the coordinate-based definition presented earlier.
\end{remark}

\subsection{Change of Coordinates in a Single Grassmannian}

We are now ready to describe the methods for defining closed subschemes of $\Gr(d,n)$ using both coordinate systems. Specifically, we aim to describe a subscheme $Z \subset \Gr(d,n)$ by an ideal $J \subset k[p]/I_{\Gr(d,n)}$ or an ideal $I \subset k[s]$, and to convert between these descriptions. Since Pl\"ucker coordinates are unique only up to a unit factor, or equivalently up to the $\GL_1$-action, the ideal $J$ must be homogeneous, or equivalently stable under the scheme-theoretic action of $\GL_1 = \mathbb{G}_m$, for the condition $J=0$ at $M$ to be well-defined. Similarly, since Stiefel matrices are unique only up to the right $\GL_d$-action, the ideal $I$ must be $\GL_d$-stable for the condition $I=0$ at $M$ to be well-defined.

\begin{definition}\label{def:vanishing_locus}
Let $Z \subset \Gr(d, n)$ be a closed subscheme.
\begin{enumerate}
\item Let $J \subset k[p]/I_{\Gr(d,n)}$ be a homogeneous ideal. We write $V_p(J) = Z$ if for every local $k$-algebra $R$ and every $M \in \Gr(d,n)(R)$, the condition $J = 0$ at $M$ holds if and only if $M \in Z(R)$.
\item Let $I \subset k[s]$ be a $\GL_d$-stable ideal. We write $V_s(I) = Z$ if for every local $k$-algebra $R$ and every $M \in \Gr(d,n)(R)$, the condition $I = 0$ at $M$ holds if and only if $M \in Z(R)$.
\end{enumerate}
\end{definition}

Since the Pl\"ucker coordinates serve as projective coordinates, $V_p(J)$ always exists. Although the existence of $V_s(I)$ is not yet guaranteed, \Cref{lem:subscheme_local_determination} implies that if such a subscheme exists, it is unique. Transforming equations from Pl\"ucker coordinates into Stiefel coordinates is straightforward.

\begin{definition}\label{def:plucker_to_stiefel}
For a polynomial $g \in k[p]/I_{\Gr(d,n)}$, let $g_s \in k[s]$ denote the polynomial obtained by substituting each indeterminate $p_\alpha$ with the minor $\det S_{\alpha,\bullet}$. For an ideal $J \subset k[p]/I_{\Gr(d,n)}$, we define $J_s \subset k[s]$ as the ideal generated by the set
\[ \{ g_s \mid g \in J\}. \]
\end{definition}

The assignment $g \mapsto g_s$ yields a well-defined $k$-algebra homomorphism $k[p] / I_{\Gr(d,n)} \to k[s]$, because the $d \times d$ minors of any $n \times d$ matrix satisfy the Pl\"ucker relations. Moreover, if $J$ is homogeneous, then $J_s$ is $\GL_d$-stable. Indeed, for a homogeneous polynomial $g \in J$, the right action of $B \in \GL_d(R)$ on $g_s$ results in multiplication by $(\det B)^{\deg g} \in R^\times$.

\begin{proposition}\label{prop:plucker_to_stiefel}
Given a homogeneous ideal $J \subset k[p]/I_{\Gr(d,n)}$, we have
\[ V_p(J) = V_s(J_s). \]
\end{proposition}
\begin{proof}
This follows immediately from the construction of $J_s$ and the definition of Pl\"ucker coordinates.
\end{proof}

Consequently, if $Z \subset \Gr(d,n)$ is defined by homogeneous polynomials $g_i \in k[p] / I_{\Gr(d,n)}$, then $Z$ is also defined by $(g_i)_s \in k[s]$. Thus, if $Z$ is defined by maximal minors of $S$, explicitly describing it in Pl\"ucker coordinates is straightforward. We now address the more general problem of converting an ideal from Stiefel coordinates to Pl\"ucker coordinates.

\begin{definition}\label{def:stiefel_to_plucker}
For $f \in k[s]$ and $\alpha \in \Inc{d}{n}$, let $f_{p,\alpha} \in k[p] / I_{\Gr(d,n)}$ be the polynomial defined by
\[ f_{p,\alpha} = f(P_\alpha) \in k[p] / I_{\Gr(d,n)}, \]
where $P_\alpha$ is the matrix defined in Definition~\ref{def:plucker_semisubmatrix}. For a $\GL_d$-stable ideal $I \subset k[s]$, we define $I_p \subset k[p] / I_{\Gr(d,n)}$ as the ideal generated by
\[ \{ f_{p,\alpha} \mid f \in I, \, \alpha \in \Inc d n \}. \]
\end{definition}

Clearly, $f \mapsto f_{p,\alpha}$ yields a homomorphism $k[s] \to k[p] / I_{\Gr(d,n)}$, and $I_p$ is homogeneous.

\begin{lemma} \label{lem:GL_dense}
Let $R$ be a ring and let $h$ be a regular function on $\Mat_{d,R}$. If $h = 0$ on $\GL_{d,R}$, then $h = 0$ on $\Mat_{d,R}$.
\end{lemma}
\begin{proof}
Let $A$ be the matrix of indeterminates for $\Mat_{d,R}$, and let $R[A]$ be its coordinate ring. The coordinate ring of $\GL_{d,R}$ is the localization $R[A]_{\det A}$. Since the determinant $\det A$ is not a zero divisor in $R[A]$, the localization map $R[A] \to R[A]_{\det A}$ is injective. Consequently, if $h = 0$ in $R[A]_{\det A}$, then $h = 0$ in $R[A]$.
\end{proof}

\begin{lemma} \label{lem:submodule_vanish}
Let $I \subset k[s]$ be a $\GL_d$-stable ideal. Let $T$ be an $n \times d$ matrix over a local $k$-algebra $R$ such that $f(T) = 0$ for all $f \in I$. Then for any $g \in I$ and for all $B \in \Mat_{d}(R)$, we have $g(TB) = 0$.
\end{lemma}
\begin{proof}
Fix $g \in I$. Let $A$ be the matrix of indeterminates for $\Mat_{d,R}$. Then $g(TA)$ defines a regular function on $\Mat_{d,R}$. Since $f(T) = 0$ for all $f \in I$ and $I$ is $\GL_d$-stable, the function $g(TA)$ vanishes on $\GL_{d,R}$. By \Cref{lem:GL_dense}, it vanishes on $\Mat_{d,R}$, which implies that $g(TB) = 0$ for all $B \in \Mat_{d}(R)$.
\end{proof}

\begin{theorem}\label{thm:stiefel_to_plucker}
Given a $\GL_d$-stable ideal $I \subset k[s]$, we have
\[ V_s(I) = V_p(I_p). \]
In particular, $V_s(I)$ is a well-defined closed subscheme of $\Gr(d,n)$.
\end{theorem}
\begin{proof}
Suppose that $M \in V_p(I_p)(R)$. Here, $M \in U_\alpha(R)$ for some $\alpha \in \Inc{d}{n}$, so we may assume that $P_\alpha = S$ by \Cref{lem:P_alpha_eq_S}. Then for every $f \in I$, we have $f_{p,\alpha} = 0$ at $M$. Since $f_{p,\alpha} = f(P_\alpha) = f(S)$, it follows that $f(S) = 0$. Hence, $I = 0$ at $M$, meaning that $V_p(I_p)(R) \subset V_s(I)(R)$.

Conversely, suppose that $M \in V_s(I)(R)$. Then for any $\alpha \in \Inc{d}{n}$, $\im P_\alpha \subset M$ by \Cref{thm:im_P_eq_M}, meaning that $P_\alpha = S \cdot B$ for some $B \in \Mat_d(R)$. \Cref{lem:submodule_vanish} now implies that
\[ f_{p,\alpha} = f(P_\alpha) = f(S B) = 0. \]
Hence, $I_p = 0$ at $M$, meaning that $V_s(I)(R) \subset V_p(I_p)(R)$.
\end{proof}

\subsection{Change of Coordinates in a Product of Grassmannians}

The results obtained in the previous subsection for a single Grassmannian can be naturally extended to products of Grassmannians. Let $N$ be a positive integer. Let $\mathbf{d} = (d_0, \dots, d_{N-1})$ and $\mathbf{n} = (n_0, \dots, n_{N-1})$ be sequences of integers such that $0 < d_i \le n_i$ for all $i \in [N]$. We define
\[
\Gr(\mathbf{d}, \mathbf{n}) \coloneqq \prod_{i=0}^{N-1} \Gr(d_i, n_i).
\]

For each $i \in [N]$, let $S^{(i)}$ denote the Stiefel matrix of size $n_i \times d_i$ representing the $i$-th factor $\Gr(d_i, n_i)$. We denote by $s^{(i)}$ the Stiefel coordinates of $\Gr(d_i, n_i)$. The polynomial ring of Stiefel coordinates for $\Gr(\mathbf{d}, \mathbf{n})$ is defined as
\[
k[s] \coloneqq \bigotimes_{i=0}^{N-1} k[s^{(i)}].
\]
Then the affine group $\GL_{\mathbf{d}} = \prod_{i=0}^{N-1} \GL_{d_i}$ acts on $k[s]$ on the right.

Similarly, let $P^{(i)}$ and $p^{(i)}$ be the Pl\"ucker matrix and Pl\"ucker coordinates of $\Gr(d_i, n_i)$, respectively. The multi-homogeneous polynomial ring for the product space $\Gr(\mathbf{d}, \mathbf{n})$ is defined as
\[
k[p] \coloneqq \bigotimes_{i=0}^{N-1} k[p^{(i)}].
\]

\begin{definition}
Let $Z \subset \Gr(\mathbf{d}, \mathbf{n})$ be a closed subscheme.
\begin{enumerate}
\item Let $J \subset k[p]/I_{\Gr(\mathbf{d}, \mathbf{n})}$ be a multi-homogeneous ideal. We write $V_p(J) = Z$ if for every local $k$-algebra $R$ and every $M \in \Gr(\mathbf{d}, \mathbf{n})(R)$, the condition $J = 0$ at $M$ holds if and only if $M \in Z(R)$.
\item Let $I \subset k[s]$ be a $\GL_{\mathbf{d}}$-stable ideal. We write $V_s(I) = Z$ if for every local $k$-algebra $R$ and every $M \in \Gr(\mathbf{d}, \mathbf{n})(R)$, the condition $I = 0$ at $M$ holds if and only if $M \in Z(R)$.
\end{enumerate}
\end{definition}

We now state the algorithms for converting ideals between Stiefel and Pl\"ucker coordinates.

\begin{definition}
Let
\[\Inc{\mathbf{d}}{\mathbf{n}} \coloneqq \prod_{i=0}^{N-1} \Inc{d_i}{n_i}.\]
\begin{enumerate}
\item (Pl\"ucker to Stiefel) For $g \in k[p]/I_{\Gr(\mathbf{d}, \mathbf{n})}$, let $g_s \in k[s]$ be the polynomial obtained by substituting each indeterminate $p_\beta^{(i)}$ with the minor $\det S_{\beta, \bullet}^{(i)}$. For an ideal $J$, we define $J_s$ as the ideal generated by $\{ g_s \mid g \in J \}$.
\item (Stiefel to Pl\"ucker) For $f \in k[s]$ and $\boldsymbol{\alpha} = (\alpha^{(0)}, \dots, \alpha^{(N-1)}) \in \Inc{\mathbf{d}}{\mathbf{n}}$, let
\[f_{p, \boldsymbol{\alpha}} \coloneqq f\!\left(P_{\alpha^{(0)}}^{(0)}, \dots, P_{\alpha^{(N-1)}}^{(N-1)}\right).\]
For a $\GL_{\mathbf{d}}$-stable ideal $I \subset k[s]$, we define $I_p$ as the ideal generated by 
\[\{ f_{p, \boldsymbol{\alpha}} \mid f \in I, \, \boldsymbol{\alpha} \in \Inc{\mathbf{d}}{\mathbf{n}} \}.\]
\end{enumerate}
\end{definition}

\begin{theorem}\label{thm:ideal_conversion}
Given a multi-homogeneous ideal $J \subset k[p]/I_{\Gr(\mathbf{d}, \mathbf{n})}$, we have $V_p(J) = V_s(J_s)$. Conversely, given a $\GL_{\mathbf{d}}$-stable ideal $I \subset k[s]$, we have $V_s(I) = V_p(I_p)$.
\end{theorem}
\begin{proof}
The first statement follows immediately from the definitions. For the second statement, the proof is analogous to that of \Cref{thm:stiefel_to_plucker}, applied to each component of the product space.
\end{proof}

\begin{remark}\label{rem:coordinate_change}
These results allow us to convert the defining equations of a subscheme between the two coordinate systems. To state this explicitly in terms of generators, we introduce the following notation. If $g_0, \dots, g_{u-1} \in k[\mathbf{p}] / I_{\Gr(\mathbf{d}, \mathbf{n})}$ generate a multi-homogeneous ideal $J$, we write
\[
V_p(g_0, \dots, g_{u-1}) \coloneqq V_p(J).
\]
Similarly, if $f_0, \dots, f_{v-1} \in k[\mathbf{s}]$ generate a $\GL_{\mathbf{d}}$-stable ideal $I$, we write
\[
V_s(f_0, \dots, f_{v-1}) \coloneqq V_s(I).
\]
We rephrase Theorem \ref{thm:ideal_conversion} in terms of generators. First, for the conversion from Pl\"ucker to Stiefel coordinates, we have
\[
V_p(g_0, \dots, g_{u-1}) = V_s((g_0)_s, \dots, (g_{u-1})_s).
\]
Second, for the conversion from Stiefel to Pl\"ucker coordinates, we have
\[
V_s(f_0, \dots, f_{v-1}) = V_p(\dots, (f_i)_{p, \boldsymbol{\alpha}}, \dots),
\]
where $\boldsymbol{\alpha}$ ranges over $\Inc{\mathbf d}{\mathbf n}$ and $i$ ranges over $[v]$.
\end{remark}

\begin{remark}
While these algorithms provide a systematic way to convert coordinates, the resulting systems of equations often contain a much larger number of polynomials than necessary, and these polynomials may have higher degrees. Consider the case of a single Grassmannian $\Gr(d,n)$. The complement of the affine chart $U_\alpha$ is defined by $V_p(p_\alpha)$. Converting this to Stiefel coordinates yields $V_s(\det S_{\alpha, \bullet})$, which is defined by a single polynomial of degree $d$. However, converting this back to Pl\"ucker coordinates results in a system generated by $(\det S_{\alpha, \bullet})_{p, \beta}$ for all $\beta \in \Inc{d}{n}$. This new system consists of $\binom{n}{d}$ polynomials, each of degree $d$.
\end{remark}

\subsection{Auxiliary Lemmas}\label{subsec:auxiliary_lemmas}

Before proceeding to the next section, we establish short lemmas regarding Grassmannians. Let $V$ be a vector space, $d \in [\dim V + 1]$ an integer, and $R$ a local $k$-algebra. Consider the Pl\"ucker embedding
\begin{align*}
\Gr(d, V)(R) &\hookrightarrow \mathbb{P}\big(\textstyle{\bigwedge^{d}} V \big)(R) \\
M &\mapsto \big[\textstyle{\bigwedge^{d}} M\big].
\end{align*}
Let $U \subset V$ be a subspace of dimension $c \leq d$. The Grassmannian $\Gr(d - c, V/U)$ parametrizes submodules $M \in \Gr(d,V)(R)$ containing $U_R$. To describe this embedding, fix a nonzero element $\omega_U \in \bigwedge^c U$. The injective linear map
\begin{align*}
\psi\colon \textstyle{\bigwedge^{d-c}} (V/U) &\to \textstyle{\bigwedge^d} V \\
\eta &\mapsto \tilde{\eta} \wedge \omega_U,
\end{align*}
where $\tilde{\eta}$ lifts $\eta$, induces the commutative diagram
\begin{equation}\label{diag:linear_section_grassmannian}
\begin{tikzcd}
\Gr(d - c, V/U) \arrow[r, hook] \arrow[d, hook]
  & \mathbb{P}\!\left(\textstyle{\bigwedge^{d-c}} (V/U)\right)
    \arrow[d, hook, "\mathbb{P}(\psi)"] \\
\Gr(d, V) \arrow[r, hook]
  & \mathbb{P}\!\left(\textstyle{\bigwedge^{d}}V \right)
\end{tikzcd}
\end{equation}
The map $\mathbb{P}(\psi)$ is well-defined and independent of the choice of $\omega_U$, so we refer to it as the \emph{natural linear embedding}. A point $p \in \mathbb{P}\big(\textstyle{\bigwedge^{d}} V\big)(R)$ lies in the image of $\mathbb{P}(\psi)$ if and only if it is divisible by $\omega_U$. Since $\omega_U$ is decomposable, this is equivalent to
\begin{equation}\label{eq:divisible_by_U}
p \wedge f = 0 \quad \text{for all } f \in U_R.
\end{equation}
\begin{remark}\label{rem:linear_embedding}
Concretely, let $e_0, \dots, e_{n-1}$ be the standard basis of $V = R^n$, and let $\{e_\alpha\}_{\alpha \in \Inc{d}{n}}$ be the induced basis of $\bigwedge^d V$. Let $f_0, \dots, f_{c-1}$ be a basis of $U$, and write $p = \sum_{\alpha} p_\alpha e_\alpha$. Then condition~\eqref{eq:divisible_by_U} becomes the system of linear equations
\[
\sum_{\alpha \in \Inc{d}{n}} p_\alpha (e_\alpha \wedge f_i) = 0 \quad \text{for all } i \in [c].
\]
This completely describes the natural linear embedding $\mathbb{P}\big(\textstyle{\bigwedge^{d-c}}(V/U)\big) \hookrightarrow \mathbb{P}\big(\textstyle{\bigwedge^{d}}V\big)$.
\end{remark}

\begin{lemma}\label{lem:linear_section_grassmannian}
In the ambient space $\mathbb{P}\big(\textstyle{\bigwedge^{d}}V\big)$, we have
\[
\Gr(d - c, V/U) = \Gr(d, V) \cap \mathbb{P}\!\left(\textstyle{\bigwedge^{d-c}} (V/U)\right).
\]
\end{lemma}
\begin{proof}
A submodule $M \in \Gr(d,V)(R)$ contains $U$ if and only if its Pl\"ucker coordinate $p$ is divisible by $\omega_U$. This is precisely the condition that $p$ lies in the linear subspace $\mathbb{P}\!\left(\textstyle{\bigwedge^{d-c}} (V/U)\right)$.
\end{proof}

To provide a variant of this result adapted for explicit computations, we introduce the following definition.

\begin{definition}\label{def:plucker_vector}
Let $R$ be a local $k$-algebra, and let $M \in \Gr(d,n)(R)$ be represented by an $n \times d$ Stiefel matrix $S$. Let $s_0, \dots, s_{d-1} \in R^n$ denote the columns of $S$. The \emph{Pl\"ucker vector} representing $M$ is the element
\[p \coloneqq s_0 \wedge \dots \wedge s_{d-1} \in \textstyle{\bigwedge^d} R^n.\]
We can also write it as
\[p = \sum_{\alpha \in \Inc{d}{n}} p_\alpha e_\alpha,\]
where $p_\alpha$ are the Pl\"ucker coordinates.
\end{definition}

\begin{lemma}\label{lem:submodule_criterion}
Let $R$ be a local $k$-algebra. Let $U \subset R^n$ be an $R$-submodule generated by the columns of a matrix
\[F = (f_0\ f_1\ \dots\ f_{c-1}) \in \Mat_{n \times c}(R).\]
Let $M \in \Gr(d,n)(R)$ be represented by a Stiefel matrix $S \in \Mat_{n \times d}(R)$, and let
\[p = \sum_{\alpha \in \Inc{d}{n}} p_\alpha e_\alpha \in \bigwedge^d R^n\]
be the associated Pl\"ucker vector. Then the following are equivalent.
\begin{enumerate}
    \item $U \subseteq M$.
    \item All $(d+1)\times(d+1)$ minors of the block matrix $(S\mid F)$ vanish.
    \item All $(d+1)\times(d+1)$ minors of $(S\mid F)$ formed by taking exactly $d$ columns from $S$ and one column from $F$ vanish.
    \item $p \wedge f_i = 0$ for all $i \in [c]$.
\end{enumerate}
\end{lemma}
\begin{proof}
This follows immediately from the theory of Fitting ideals \cite{Fitting}\cite[Section~20.2]{Eisenbud2}. These conditions are equivalent to the vanishing of a $(n-d-1)$-th Fitting ideal:
\[\Fitt_{n-d-1}\bigl(R^n/(M+U)\bigr) = 0. \qedhere\]
\end{proof}

\begin{remark}
    By \Cref{lem:submodule_criterion}, in the Stiefel coordinates of $M$, the condition $U \subset M$ is defined by homogeneous polynomials of degree $d$. In the Pl\"ucker coordinates, however, the condition is defined by linear homogeneous polynomials.
\end{remark}

\section{Numerical Conditions}\label{sec:numerical}

Throughout the paper, let $X \hookrightarrow \mathbb{P}^r$ be a smooth connected projective variety over a field $k$, defined by homogeneous polynomials of degree at most $\delta$. Moreover, let $H \subset X$ denote a hyperplane section of $X$. In this section, we relax the assumption that $R$ is local, so $R$ denotes an arbitrary $k$-algebra. 

As previously mentioned, the main objective of this paper is to compute $\Pic^\tau X$ and its group scheme structure. Since the twisting morphism $\mathcal{L} \mapsto \mathcal{L}(m)$ induces an isomorphism $\Pic^\tau X \xrightarrow{\sim} \Pic_{mH} X$, it suffices to compute the subscheme $\Pic_{mH} X$. Because the class map given by
\begin{align*}
(\Div_{mH} X)(R) &\to (\Pic_{mH} X)(R) \\
[D] &\mapsto \mathcal{O}_{X_R}(D)
\end{align*}
is a projective bundle for sufficiently large $m$, we will construct $\Pic_{mH} X$ as a quotient of $\Div_{mH} X$. Furthermore, since the map
\begin{align*}
(\Div_{mH} X)(R) &\hookrightarrow \Gr(P_X(t-m), (S_X)_t)(R) \\
[D] &\mapsto (I_{D/X_R})_t
\end{align*}
induces a closed embedding for sufficiently large $t$, we will construct $\Div_{mH} X$ as a closed subscheme of the Grassmannian. 

Each of these steps depends on the integers $m$ and $t$ being sufficiently large. The purpose of this section is to provide explicit, computable lower bounds for $m$ and $t$ that guarantee these properties.

\subsection{The Projective Bundle \texorpdfstring{$\Div_{mH} X \to \Pic_{mH} X$}{Div\_mH X -> Pic\_mH X}}

We now seek to determine a lower bound for $m$ such that the morphism $\Div_{mH} X \to \Pic_{mH} X$ becomes a projective bundle. To state this bound, we recall the definition of the Gotzmann number.

\begin{definition}
A coherent sheaf $\mathcal{F}$ on $\mathbb{P}^r$ is said to be \emph{$m$-regular} if 
\[
H^i\!\left(\mathbb{P}^r, \mathcal{F}(m - i)\right) = 0
\]
for all integers $i > 0$. The smallest integer $m$ with this property is called the \emph{Castelnuovo–Mumford regularity} of $\mathcal{F}$.
\end{definition}

If a coherent sheaf $\mathcal{F}$ is $m$-regular, then it is also $\mu$-regular for every $\mu \ge m$ by \cite[p.~99]{Mumford}.

\begin{definition}
Given a polynomial $Q(s)$, the \emph{Gotzmann number} $\varphi(Q)$ of $Q$ is defined by
\begin{align*}
\varphi(Q) = \inf \{ m \,|\, &
\mathscr{I}_Z \text{ is $m$-regular for every closed subscheme } \\
& 
Z \subset \mathbb{P}^r \text{ with Hilbert polynomial } Q \}.
\end{align*}
For a projective scheme $Z \subset \mathbb{P}^r$, we write $\varphi(Z)$ for the Gotzmann number of the Hilbert polynomial of $\mathscr{I}_Z$.
\end{definition}

\begin{remark}
Note that $\varphi(Q)$ and $\varphi(Z)$ depend on the dimension $r$ of the ambient space, which is fixed throughout the paper. The Gotzmann number $\varphi(Q)$ is explicitly computable. Recall that any numerical polynomial $Q(s)$ can be uniquely written in the Macaulay representation
\[
Q(s) = \binom{s+a_1}{a_1} + \binom{s+a_2-1}{a_2} + \dots + \binom{s+a_\psi-(\psi-1)}{a_\psi},
\]
where $a_1 \ge a_2 \ge \dots \ge a_\psi \ge 0$ are integers \cite[Remark~C.11]{IaKa}. The coefficients $a_i$ are uniquely determined and can be computed via a simple greedy algorithm by setting $a_1$ to be the degree of $Q(s)$ and inductively applying the same procedure to the remainder $Q(s) - \binom{s+a_1}{a_1}$. The Gotzmann number is given by $\varphi(Q) = \psi$ \cite[Proposition~C.24]{IaKa}. In particular, if $Q$ is the Hilbert polynomial of a homogeneous ideal, then $\varphi(Q) > 0$. If $Q$ is not the Hilbert polynomial of a homogeneous ideal, then by definition $\varphi(Q)=\infty$.
\end{remark}

\begin{remark}
For a closed subscheme $Z \subset \mathbb{P}^r$, the Hilbert polynomial $Q_Z(s)$ can be computed algorithmically \cite[Algorithm~2.7]{MoraMoller}\cite[Algorithm~2.6]{BayerStillman}. Moreover, Hoa \cite[Theorem~6.4(i)]{Hoa} provided an explicit upper bound for the Gotzmann number $\varphi(Z)$.
\end{remark}

\begin{theorem}\label{thm:div_to_pic}
If $\nu = (\delta-1)\codim X$ and $m \ge \max\{ \varphi(\nu H), \varphi(X) \}$, then the class map
\[ \Div_{mH} X \to \Pic_{mH} X \]
is a projective bundle. Moreover, its fibers have constant dimension $P_X(m) - 1$.
\end{theorem}

\begin{proof}
Let $\mathcal{L}$ be a line bundle on $X$ numerically equivalent to $\mathcal{O}_X(m)$. By \cite[Lemma~3.5]{Kwe},
\[ H^i(X, \mathcal{L}) = 0 \]
for every $i \ge 1$. Hence, by \cite[Theorem~3.1]{Kwe} and the isomorphism $\Pic^\tau X \xrightarrow{\sim} \Pic_{mH} X$, the class map $\Div_{mH} X \to \Pic_{mH} X$ is a projective bundle.\footnote{Although the statement of \cite[Theorem~3.1]{Kwe} asserts only faithful flatness of $\Div_{mH} X \to \Pic^\tau X$, the proof shows that the morphism is in fact a projective bundle.}

To determine the fiber dimension, we change the base field $k$ to its algebraic closure. This ensures that every connected component of $\Pic_{mH} X$ contains a rational point, corresponding to a line bundle $\mathcal{L}$ on $X$ numerically equivalent to $\mathcal{O}_X(m)$. The fiber over $\mathcal{L}$ is isomorphic to $\mathbb{P}(H^0(X, \mathcal{L}))$. By the Hirzebruch--Riemann--Roch theorem, the Euler characteristic of a line bundle is a numerical invariant. Therefore,
\[ \chi(X, \mathcal{L}) = \chi(X, \mathcal{O}_X(m)). \]
Moreover, since $m \ge \varphi(X)$, we have
\[ H^i(X, \mathcal{O}_X(m)) = 0 \]
for every $i \ge 1$. Consequently,
\[ \dim H^0(X, \mathcal{L}) = \chi(X, \mathcal{L}) = \chi(X, \mathcal{O}_X(m)) = P_X(m). \qedhere \]
\end{proof}

\subsection{Embedding the Hilbert Scheme into a Grassmannian}

We now explain how $\Div_{mH} X$ can be embedded into a large Grassmannian, which allows us to handle this moduli space explicitly. With $Q = Q_{mH}$, consider the sequence of natural maps
\[
\Div_{mH} X \hookrightarrow \Div^Q X \hookrightarrow \Hilb^Q X \hookrightarrow \Hilb^Q \mathbb{P}^r.
\]
The first map is well-defined since the Hilbert polynomial is a numerical invariant. All of these maps are closed embeddings, and the first two are also open immersions. Thus, this sequence allows us to utilize the embedding of $\Hilb^Q \mathbb{P}^r$ into a Grassmannian.

\begin{theorem}[Gotzmann]\label{thm:gotzmann}
Let $Q$ be a polynomial. If $t \ge \varphi(Q)$, then there exists a natural closed embedding
\begin{align*}
\Hilb^Q \mathbb{P}^r &\hookrightarrow \Gr(Q(t), S_t) \\
[Z] &\mapsto (I_Z)_t.
\end{align*}
\end{theorem}
\begin{proof}
This follows from \cite{Got}; see \cite[Proposition~C.29]{IaKa} for a modern exposition.
\end{proof}

If $Q$ is not the Hilbert polynomial of any ideal sheaf $\mathscr{I}_Z \subset \mathcal{O}_{\mathbb{P}^r}$, then $\varphi(Q) = \infty$, and the theorem holds vacuously. Moreover, for $t \ge \varphi(Q)$, we may also regard $\Hilb^Q X$ as a closed subscheme of the Grassmannian $\Gr(Q(t), S_t)$. We now provide an alternative description of the embedding of $\Hilb^Q X$ into a Grassmannian.

\begin{lemma}\label{lem:hilb_X}
Let $Q$ be a polynomial, and assume that $t \ge \max\{\varphi(Q),\varphi(X)\}$. We regard $\Hilb^Q X$ and $\Hilb^Q \mathbb{P}^r$ as closed subschemes of $\Gr(Q(t), S_t)$. Then an $R$-submodule $M \in (\Hilb^Q \mathbb{P}^r)(R)$ of $(S_R)_t$ lies in $(\Hilb^Q X)(R)$ if and only if
\[(I_{X_R})_t \subseteq M.\]
\end{lemma}
\begin{proof}
Let $Z \subset \mathbb{P}^r_R$ be the closed subscheme defined by $M$, and let $I_M \subset S_R$ be the ideal generated by $M$. Then $Z \subset X_R$ if and only if
\[I_{X_R} \subset I_M.\]
Since $t \ge \varphi(X)$, the ideal sheaf $\mathscr I_X$ is generated by $(I_X)_t$. As a result, the condition above is equivalent to
\[(I_{X_R})_t \subseteq M. \qedhere\]
\end{proof}

\begin{theorem}\label{thm:hilb_X_embedding}
Let $Q$ be a polynomial, and assume that $t \ge \max\{\varphi(Q), \varphi(X)\}$. Then there exists a natural closed embedding
\begin{align*}
\Hilb^Q X &\hookrightarrow \Gr\big(Q(t) - Q_X(t), (S_X)_t\big) \\
[Z] &\mapsto (I_{Z/X})_t.
\end{align*}
\end{theorem}
\begin{proof}
Since $t \ge \varphi(X)$, we have $H^1(\mathbb{P}^r, \mathscr{I}_X(t)) = 0$. The short exact sequence
\[
0 \to \mathscr{I}_X \to \mathcal{O}_{\mathbb{P}^r} \to j_* \mathcal{O}_X \to 0,
\]
where $j \colon X \to \mathbb{P}^r$ is the closed embedding, induces the exact sequence
\[
0 \to (I_X)_t \to S_t \to (S_X)_t \to 0.
\]
Thus, we may identify $(S_X)_t$ with the quotient $S_t / (I_X)_t$. The result now follows from \Cref{thm:gotzmann} and \Cref{lem:hilb_X}.
\end{proof}

Hence, \Cref{lem:hilb_X} and the diagram (\ref{diag:linear_section_grassmannian}) imply the following commutative diagram.

\begin{equation}\label{diag:linear_section_hilbert}
\begin{tikzcd}
\Hilb^Q X \arrow[r, hook] \arrow[d, hook]
  & \Gr\big(Q(t) - Q_X(t), (S_X)_t\big) \arrow[r, hook] \arrow[d, hook]
  & \mathbb{P}\!\left(\bigwedge^{Q(t) - Q_X(t)} (S_X)_t \right)
    \arrow[d, hook, "\mathbb{P}(\psi)"]\\
\Hilb^Q \mathbb{P}^r \arrow[r, hook]
  & \Gr(Q(t), S_t) \arrow[r, hook]
  & \mathbb{P}\!\left(\bigwedge^{Q(t)} S_t\right)
\end{tikzcd}
\end{equation}

\begin{lemma}\label{lem:linear_section_hilbert}
Let $Q$ be a polynomial, and assume that $t \ge \max\{\varphi(Q), \varphi(X)\}$. Then in the ambient space $\mathbb{P}(\bigwedge^{Q(t)} S_t)$, we have
\[\Hilb^Q X = \Hilb^Q \mathbb{P}^r \cap \mathbb{P}\!\left(\textstyle{\bigwedge^{Q(t) - Q_X(t)}} (S_X)_t \right).\]
\end{lemma}
\begin{proof}
\Cref{lem:hilb_X} implies that
\[\Hilb^Q X = \Hilb^Q \mathbb{P}^r \cap \Gr\big(Q(t) - Q_X(t), (S_X)_t\big).\]
The result now follows from \Cref{lem:linear_section_grassmannian}.
\end{proof}

\begin{remark}
If $Q = Q_{mH}$, we have $Q(s) - Q_X(s) = P_X(s-m)$. This follows from the short exact sequence
\[
0 \to \mathscr{I}_X \to \mathscr{I}_{mH} \to j_* \mathscr{I}_{mH/X} \to 0,
\]
where $j \colon X \hookrightarrow \mathbb{P}^r$ denotes the closed embedding. Since $\mathscr{I}_{mH/X} \cong \mathcal{O}_X(-m)$, the additivity of Hilbert polynomials implies
\[Q_X(s) - Q_{mH}(s) + P_X(s-m) = 0.\]
\end{remark}

\subsection{Numerical Hypotheses on \texorpdfstring{$m$}{m} and \texorpdfstring{$t$}{t}}

We now summarize the numerical assumptions required for the remainder of the paper. Let $\delta$ be the maximum degree of the defining equations of $X \hookrightarrow \mathbb{P}^r$. We define an auxiliary integer $\nu \coloneqq (\delta-1)\codim X$. The results of this section hold under the following conditions.
\[
    m \ge \max \{\varphi(\nu H), \varphi(X)\} \quad \text{and} \quad t \ge \max \{\varphi(mH), \varphi(X)\}
\]
Specifically, the following properties hold.
\begin{enumerate}
    \item The quotient map $\Div_{mH} X \twoheadrightarrow \Pic_{mH} X$ is a projective bundle (\Cref{thm:div_to_pic}).
    \item The embeddings of Hilbert schemes into Grassmannians fit into a commutative diagram (\ref{diag:linear_section_hilbert}).
\end{enumerate}
These conditions are sufficient to compute the moduli space $\Pic^\tau X$. However, the addition on $\Pic^\tau X$ is constructed as a quotient of the addition morphism of divisors $\sigma \colon \Div_{mH} X \times \Div_{mH} X \to \Div_{2mH} X$. This requires us to apply the results of this section to $\Div_{2mH} X$ as well. Therefore, throughout the rest of the paper, we assume the following stronger conditions.
\begin{equation}
    m \ge \max \{\varphi(\nu H), \varphi(X)\} \quad \text{and} \quad t \ge \max \{\varphi(2mH), \varphi(X)\}
    \label{cond:big_numbers}\tag{$\ast$}
\end{equation}
Note that the parameters $\delta$ and $\nu$ are introduced solely to define the bounds for $m$ and $t$, and will not be explicitly used in the subsequent sections.

\section{Computing \texorpdfstring{$\Div_{mH} X$}{Div\_mH X}}\label{sec:Div_mH_X}

The aim of this section is to present an algorithm that computes explicit homogeneous equations defining $\Div_{mH} X$ as a closed subscheme of a projective space. With $Q = Q_{mH}$, recall the following commutative diagram of embeddings.
\begin{equation}\label{diag:div_hilb_grass_proj}
\begin{tikzcd}[column sep=0.925em]
\Div_{mH} X \arrow[r, hook] &
\Div^Q X \arrow[r, hook] &
\Hilb^Q X \arrow[r, hook] \arrow[d, hook] &
\Gr(P_X(t-m),(S_X)_t) \arrow[r, hook] \arrow[d, hook] &
\mathbb{P}\!\left(\textstyle{\bigwedge^{P_X(t-m)}} (S_X)_t\right) \arrow[d, hook] \\
& &
\Hilb^Q \mathbb P^r \arrow[r, hook] & 
\Gr(Q(t),S_t) \arrow[r, hook] &
\mathbb{P}\!\left(\textstyle{\bigwedge^{Q(t)}} S_t\right)
\end{tikzcd}
\end{equation}

As we move towards the top-left in the diagram, additional defining polynomials are required to cut out each subsequent subscheme. Eventually, we will describe the sequential computation of these polynomials to establish the embedding
\[ \Div_{mH} X \hookrightarrow \mathbb{P}\!\left(\textstyle{\bigwedge^{P_X(t-m)}} (S_X)_t\right). \] 

\subsection{Explicit Equations for the Hilbert Scheme}

In this subsection, we present an algorithm for computing the embeddings appearing in the two rightmost squares of diagram \eqref{diag:div_hilb_grass_proj}. Among these, the only embedding that has not yet been discussed is the closed embedding
\[\Hilb^Q \mathbb{P}^r \hookrightarrow \Gr(Q(t), S_t),\]
together with the equations in Pl\"ucker coordinates that cut it out. Although this topic has been extensively studied in the literature, we briefly review it here in order to introduce the notation that will be used later. There are several approaches to describing the equations defining the Hilbert scheme. For simplicity, we follow the approach of Iarrobino and Kleiman \cite[Appendix~C]{IaKa}.

We denote the $(\dim S_t) \times Q(t)$ Stiefel matrix of $\Gr(Q(t), S_t)$ by
\[\Omega = (f_i)_{i \in [Q(t)]}.\]
Here, we regard each $f_i$ both as a column vector of $\Omega$ and as a homogeneous polynomial in $S_t$. Since we view the $f_i$ as polynomials, we may form another polynomial $x_j f_i \in S_{t+1}$ whose coefficients are linear in the Stiefel coordinates. By regarding each $x_j f_i$ as a column vector, we obtain a $(\dim S_{t+1}) \times (r+1)Q(t)$ matrix
\[\widehat{\Omega} = (x_j f_i)_{(i,j) \in [Q(t)]\times[r+1]}.\]

\begin{lemma}\label{lem:hilb_eq}
Let $R$ be a local $k$-algebra, and let $M \in \Gr(Q(t), S_t)(R)$ be a free $R$-module. Then $M \in (\Hilb^Q \mathbb{P}^r)(R)$ if and only if all minors of size $Q(t+1)+1$ of the matrix $\widehat{\Omega}$ vanish at $M$.
\end{lemma}
\begin{proof}
The proof is given in the proof of \cite[Theorem~C.30]{IaKa}; see also \cite[Section~4, pp.~754--755]{HaSt}.
\end{proof}

Consequently, under the embedding
\[\Hilb^Q\mathbb P^r \hookrightarrow \Gr(Q(t),S_t),\]
$\Hilb^Q \mathbb{P}^r$ is cut out by homogeneous polynomials of degree $Q(t+1) + 1$ in Stiefel coordinates. Moreover, since the column space of $\Omega$ is stable under the $\GL_{Q(t)}$-action, the column space of $\widehat{\Omega}$ is also $\GL_{Q(t)}$-stable. Thus, the defining ideal is stable under the action of $\GL_{Q(t)}$. Therefore, by \Cref{rem:coordinate_change}, the Hilbert scheme $\Hilb^Q \mathbb{P}^r$ is cut out by homogeneous polynomials of degree $Q(t+1) + 1$ in Pl\"ucker coordinates.

The embedding $\Gr(Q(t), S_t) \hookrightarrow \mathbb{P}(\textstyle{\bigwedge^{Q(t)}} S_t)$ is defined by the Pl\"ucker relations, which allows us to compute the bottom row of Diagram~\eqref{diag:div_hilb_grass_proj}. The rightmost vertical map in the diagram is determined by \Cref{rem:linear_embedding}. The remaining embeddings in the two rightmost squares are obtained via \Cref{lem:linear_section_hilbert} and \Cref{lem:linear_section_grassmannian}. This leads to the following corollary.

\begin{corollary}\label{cor:hilb_X}
    There exists an algorithm to compute the Hilbert scheme $\Hilb^Q X$ under the embeddings
    \[ \Hilb^Q X \hookrightarrow \mathbb{P}\!\left(\textstyle{\bigwedge^{P_X(t-m)}} (S_X)_t\right). \] 
\end{corollary}

\begin{remark}
The defining equations derived from the rank condition on $\widehat{\Omega}$ are known as the Iarrobino--Kleiman equations. Both in Stiefel and Pl\"ucker coordinates, these equations have degree $Q(t+1)+1$, which is computationally prohibitive. Bayer \cite[Chapter~VI]{Bayer} proposed alternative equations of degree $r+1$ in Pl\"ucker coordinates. While these forms appear as higher-degree multiples of the Iarrobino--Kleiman polynomials in Stiefel coordinates, they are polynomials in the maximal minors of the Stiefel matrix. Thus, their degree reduces significantly to $r+1$ when expressed in Pl\"ucker coordinates. Although Bayer did not prove that these equations define the correct scheme structure, their scheme-theoretic correctness was later established by Haiman and Sturmfels \cite[Section~4]{HaSt}.

We also mention that Gotzmann \cite[Section~3]{Got} provided the first concrete description of the Hilbert scheme using his Persistence Theorem. However, we do not employ his approach here, as it characterizes the Hilbert scheme as a closed subscheme of a product of two Grassmannians. For another efficient description in characteristic zero, we refer the reader to \cite{BPMR}.
\end{remark}

\subsection{Identifying the Divisor Component}

We now proceed to compute the remaining series of embeddings
\[ \Div_{mH} X \hookrightarrow \Div^Q X \hookrightarrow \Hilb^Q X. \]
A key observation here is that both embeddings are open and closed.

\begin{lemma}
The subschemes $\Div^Q X$ and $\Div_{mH} X$ are both open and closed in the ambient space $\Hilb^Q X$.
\end{lemma}
\begin{proof}
First, the embedding $\Div^Q X \hookrightarrow \Hilb^Q X$ is open and closed by \cite[Theorem~1.13]{Kol}. It remains to show that the inclusion $\iota \colon \Div_{mH} X \hookrightarrow \Div^Q X$ is open and closed.

Consider the base change to the algebraic closure $\bar{k}$. The scheme $\Div_{mH} X_{\bar{k}}$ consists of the union of connected components of $\Div^Q X_{\bar{k}}$ corresponding to the algebraic equivalence classes numerically equivalent to $mH$. Thus, the inclusion $\iota_{\bar{k}}$ is open and closed. Since the property of being an open and closed immersion satisfies fpqc descent \cite[Proposition~2.6.2]{EGA4_2}, the map $\iota$ is open and closed over $k$.
\end{proof}

As established in the previous lemma, $\Div_{mH} X$ is a union of connected components of $\Div^Q X$, which is in turn a union of connected components of $\Hilb^Q X$. Consequently, the problem reduces to collecting the specific connected components of $\Hilb^Q X$ that correspond to effective divisors numerically equivalent to $mH$. Thus, it suffices to sample a closed point from each irreducible component of $\Hilb^Q X$ and test whether the corresponding subscheme belongs to $\Div_{mH} X$.

\begin{lemma}\label{lem:point_sampling}
Let $Y \subset \mathbb{P}^r$ be a closed subscheme defined over a field $k$. There exists an algorithm that selects exactly one closed point from each irreducible component of $Y$.
\end{lemma}
\begin{proof}
Compute the primary decomposition of the defining ideal $I_Y$ to obtain the ideals of the irreducible components. Let $Z$ be such a component. We iteratively intersect $Z$ with hyperplanes, choosing each new hyperplane to be transverse to the intersection of the previous ones. When the resulting intersection $W$ becomes zero-dimensional, the associated primes of $I_W$ correspond to closed points of $Z$, one of which can be selected algorithmically.
\end{proof}

A sampled closed point of $\Hilb^Q X$ with residue field $L$ corresponds to a closed subscheme $Z \subset X_L$. The Pl\"ucker coordinates of this point in $\mathbb{P}\!\left(\textstyle{\bigwedge^{Q(t)}} S_t\right)$ explicitly determine a basis for the $L$-vector space $(I_Z)_t$. Since $t \ge \varphi(Z)$, the ideal generated by this basis defines the subscheme $Z \subset X_L$.

\begin{lemma}
Given a finite set of generators for an ideal $I_Z \subset S$ of a closed subscheme $Z$ of $X$, there exists an algorithm to determine whether $Z$ is an effective Cartier divisor on $X$.
\end{lemma}
\begin{proof}
Recall that on a smooth variety $X$, a closed subscheme is an effective Cartier divisor if and only if it has pure codimension one. To verify this, we first compute the primary decomposition of the ideal $I_Z$. For each associated prime $\mathfrak{p}$, we compute the Hilbert polynomial of the coordinate ring $S/\mathfrak{p}$ to determine its Krull dimension. The subscheme $Z$ is a divisor if and only if the Krull dimension is equal to $\dim X - 1$ for every associated prime $\mathfrak{p}$.
\end{proof}

\begin{lemma}
Given a divisor $D$ on $X$, there exists an algorithm to determine whether $D \equiv mH$.
\end{lemma}
\begin{proof}
There is an explicit upper bound $N$ for the order of the torsion subgroup of $\NS X$, depending only on the degree $\delta$ and the dimension $r$ of $X$ \cite[Theorem~4.12]{Kwe}. Recall that $D \equiv mH$ if and only if the class of $D - mH$ is a torsion element in $\NS X$. This is equivalent to the condition that the class of $N!(D - mH)$ is zero in $\NS X$, or equivalently, that $N!D$ and $N!mH$ are algebraically equivalent.

Let $Q = Q_{N!mH}$. Each connected component of $\Div^Q X$ corresponds to a single element of $\NS X$, representing an algebraic equivalence class. Thus, the problem reduces to checking whether $[N! D]$ and $[N! mH]$ belong to the same connected component of $\Div^Q X$. As the connected components of $\Div^Q X$ are algorithmically computable by primary decomposition, the result follows.
\end{proof}

These results allow us to identify the connected components of $\Hilb^Q X$ corresponding to $\Div_{mH} X$. Combining this with the computation of $\Hilb^Q X$ presented in the previous subsection, we conclude this section by stating the following theorem.

\begin{theorem}\label{thm:compute_div}
There exists an algorithm to compute the defining homogeneous equations of the closed embedding
\[ \Div_{mH} X \hookrightarrow \mathbb{P}\!\left(\textstyle{\bigwedge^{P_X(t-m)}} (S_X)_t\right). \] 
\end{theorem}

\section{Computing \texorpdfstring{$\Pic^\tau X$}{Pic\textasciicircumτ X}}\label{sec:Pic_tau_X}

The goal of this section is to present an algorithm for computing $\Pic^\tau X \cong \Pic_{mH} X$ and its group scheme structure. As mentioned previously, this scheme is constructed as a quotient of $\Div_{mH} X$. This quotient is defined by the scheme-theoretic linear equivalence relation
\[
\mathbf{L} \coloneqq \Div_{mH} X \times_{\Pic_{mH} X} \Div_{mH} X \hookrightarrow \Div_{mH} X \times \Div_{mH} X.
\]
Let $R$ be a local $k$-algebra in the rest of the paper. Then a pair $([D],[E]) \in (\Div_{mH} X \times \Div_{mH} X)(R)$ lies in $\mathbf{L}(R)$ if and only if $D$ and $E$ are linearly equivalent.

We now explain how to form the quotient of $\Div_{mH} X$ by the relation $\mathbf{L}$. Consider the projections to the $i$-th factors
\[
\pi_i\colon \mathbf{L} \to \Div_{mH} X
\]
for $i = 0, 1$. The morphism $\pi_1$ is a base change of the structure morphism $\Div_{mH} X \to \Pic_{mH} X$. Since the latter is a projective bundle, $\pi_1$ is flat. Thus, $\pi_1$ induces a morphism
\[
\phi\colon \Div_{mH} X \to \Hilb(\Div_{mH} X),
\]
which sends $[D] \in (\Div_{mH} X)(R)$ to the fiber
\[
\mathbf{L}_{[D]} \coloneqq \pi_0(\pi_1^{-1}(D)) \hookrightarrow (\Div_{mH} X)_R \hookrightarrow \mathbb{P}\left(\textstyle{\bigwedge^{P_X(t-m)}} (S_X)_t \right)_R.
\]

\begin{lemma}\label{lem:phi_image}
    The image of the morphism
    \[\phi\colon \Div_{mH} X \to \Hilb(\Div_{mH} X)\]
    is isomorphic to $\Pic_{mH} X$.
\end{lemma}
\begin{proof}
  This follows from the arguments in \cite[Lemma~9.4.9]{Kle} and \cite[Theorem~2.9]{AltmanKleiman1980}.
\end{proof}

Thus, the goal of this section is to explicitly compute the image of the morphism $\phi$. Since the full Hilbert scheme $\Hilb(\Div_{mH} X)$ consists of infinitely many connected components, we first prove that the image of $\phi$ lies within the connected components corresponding to a single Hilbert polynomial $\Phi_{m,t}$. Subsequently, we explicitly compute the linear equivalence relation $\mathbf{L}$ and the graph of $\phi$, and finally determine the defining equations of the image of $\phi$ via projective elimination theory. Lastly, the group structure is recovered by quotienting the addition morphism $\sigma \colon \Div_{mH} X \times \Div_{mH} X \to \Div_{2mH} X$.

\subsection{Hilbert Polynomial of the Fibers}

Since the Hilbert polynomial is invariant under base change, we may assume that the base field $k$ is algebraically closed throughout this subsection. Let $[D] \in (\Div_{mH} X)(k)$ be a closed point, which corresponds to an effective divisor $D$ on $X$. The fiber of the projection $\pi_1\colon \mathbf{L}\to \Div_{mH} X$ over $[D]$ is isomorphic to the complete linear system
\[
\mathbf{L}_{[D]} \cong \mathbb{P}(H^0(X, \mathcal{O}_X(D))).
\]
Therefore, the Hilbert polynomial of the connected component containing $\phi([D])$ is identical to the Hilbert polynomial of the image of the embedding
\begin{equation}\label{eq:linear_system_embedding}
\mathbb{P}(H^0(X, \mathcal{O}_X(D))) \hookrightarrow \mathbb{P}\left(\textstyle{\bigwedge^{P_X(t-m)}} (S_X)_t \right).
\end{equation}
Thus, our task reduces to computing the Hilbert polynomial of this embedding. First, we show that the Hilbert polynomial of an embedding of a projective space is described by a simple formula.

\begin{lemma}\label{lem:hilb_proj}
Let $j \colon \mathbb{P}^\mu \to \mathbb{P}^\nu$ be an embedding defined by homogeneous polynomials of degree $\delta$. Then the Hilbert polynomial of the image $Y = j(\mathbb{P}^\mu)$ is
\[
P_Y(s) = \binom{\delta s + \mu}{\mu}.
\]
\end{lemma}
\begin{proof}
Note that $j^*\mathcal{O}_Y(1) \cong \mathcal{O}_{\mathbb{P}^\mu}(\delta)$. Thus, for $s \ge 0$,
\[
P_Y(s) = \dim H^0(Y, \mathcal{O}_Y(s)) = \dim H^0(\mathbb{P}^\mu, \mathcal{O}_{\mathbb{P}^\mu}(\delta s)) = \binom{\delta s + \mu}{\mu}. \qedhere
\]
\end{proof}

By \Cref{thm:div_to_pic}, the dimension of the linear system $\mathbb{P}(H^0(X, \mathcal{O}_X(D)))$ is $\mu \coloneqq P_X(m) - 1$. To determine the Hilbert polynomial explicitly, it suffices to compute the degree of the embedding \eqref{eq:linear_system_embedding}. Observe that this embedding factors through the Grassmannian via the map
\begin{align*}
\mathbb{P}(H^0(X, \mathcal{O}_X(D))) &\to \Gr(P_X(t-m), (S_X)_t) \\
f &\mapsto f \cdot (I_{D/X})_t.
\end{align*}

\begin{lemma}\label{lem:map_degree}
The map between projective spaces
\begin{align*}
\mathbb{P}(H^0(X, \mathcal{O}_X(D))) &\to \mathbb{P}\left(\textstyle{\bigwedge^{P_X(t-m)}} (S_X)_t \right)\\
f &\mapsto \textstyle \bigwedge^{P_X(t-m)} f \cdot (I_{D/X})_t
\end{align*}
is homogeneous of degree $P_X(t-m)$.
\end{lemma}
\begin{proof}
Let $\delta = P_X(t-m)$ and let $\{g_0, \dots, g_{\delta-1}\}$ be a basis of $(I_{D/X})_t$. 
The morphism in the statement is the projectivization of the map given by
\begin{align*}
H^0(X, \mathcal{O}_X(D)) &\to \textstyle{\bigwedge^{P_X(t-m)}} (S_X)_t\\
f &\mapsto (f g_0) \wedge (f g_1) \wedge \dots \wedge (f g_{\delta-1}).
\end{align*}
Since the wedge product is multilinear, for any scalar $\lambda \in k$, we have
\[
(\lambda f g_0) \wedge \dots \wedge (\lambda f g_{\delta-1}) = \lambda^{\delta} (f g_0 \wedge \dots \wedge f g_{\delta-1})
\]
which shows that the underlying map is homogeneous of degree $\delta$.
Therefore, the morphism is defined by homogeneous polynomials of degree $\delta = P_X(t-m)$.
\end{proof}

Combining these observations, we determine the Hilbert polynomial corresponding to the image of $\phi$.

\begin{theorem}\label{lem:hilb_Phi}
The image of the morphism $\phi \colon \Div_{mH} X \to \Hilb(\Div_{mH} X)$ lies in the connected components corresponding to the Hilbert polynomial
\[
\Phi_{m,t}(s) \coloneqq \binom{P_X(t-m)s + P_X(m) - 1}{P_X(m) - 1}.
\]
\end{theorem}
\begin{proof}
This follows from applying \Cref{lem:hilb_proj} with $\mu = P_X(m) - 1$ and $\delta = P_X(t-m)$.
\end{proof}

Consequently, setting $\Phi = \Phi_{m,t}$, we may restrict the codomain of $\phi$ and write
\[
\phi \colon \Div_{mH} X \to \Hilb_{\Phi}(\Div_{mH} X).
\]

\subsection{Computing the Quotient}

We are now ready to compute $\Pic^\tau X$, which is the first main objective of this paper. This is achieved by computing the image of the morphism $\phi$. The first step is to provide an algorithm for computing
\[
\mathbf{L} = \Div_{mH} X \times_{\Pic_{mH} X} \Div_{mH} X \hookrightarrow \Div_{mH} X \times \Div_{mH} X.
\]
As usual, $R$ denotes an arbitrary local $k$-algebra.

\begin{lemma}\label{lem:linear_equiv}
  Let $D$ and $E$ be relative divisors such that $[D], [E] \in (\Div_{mH} X)(R)$. Then $D$ and $E$ are linearly equivalent if and only if
    \[
        D - E = \divisor p - \divisor q
    \]
    for some homogeneous polynomials $p, q \in (S_X)_t \otimes_k R$ of degree $t$ such that $p$ and $q$ each have at least one coefficient that is a unit in $R$.
\end{lemma}
\begin{proof}
    Suppose that $D$ and $E$ are linearly equivalent. Since $R$ is local, we choose an element $p \in (I_{D/X})_t$ corresponding to a column of a Stiefel matrix representing $[D]$. Since the Stiefel matrix possesses an invertible maximal minor, we ensure that $p$ has at least one coefficient that is a unit in $R$. By \cite[Proposition~1.11]{Kol}, such a $p$ defines a relative effective Cartier divisor. Since $p$ is homogeneous of degree $t$, we have $\divisor p \sim tH$. We write
    \[
        \divisor p = D + C
    \]
    for some relative effective divisor $C$. Since $D \sim E$, it follows that
    \[
        E + C \sim D + C = \divisor p \sim tH.
    \]
    Thus, there exists a polynomial $q \in (S_X)_t \otimes_k R$ such that $\divisor q = E + C$. Since $E+C$ is a relative effective Cartier divisor, \cite[Proposition~1.11]{Kol} implies that the reduction of $q$ to the special fiber is not a zero divisor. As $R$ is local, this guarantees that at least one coefficient of $q$ is a unit. We conclude that
    \[
        \divisor p - \divisor q = (D+C) - (E+C) = D - E. \qedhere
    \]
\end{proof}

The conditions on $p$ and $q$ in \Cref{lem:linear_equiv} correspond exactly to the conditions for projective coordinates of $R$-points of $\mathbb{P}((S_X)_t)$.

\begin{lemma}\label{lem:linear_equivalence_pq}
    There exists an algorithm to compute the defining homogeneous equations in Pl\"ucker coordinates for the closed subscheme $\mathbf{W}$ of 
    \[ 
        \mathbf{H} \coloneqq \mathbb{P}((S_X)_t) \times \mathbb{P}((S_X)_t) \times \Div_{mH} X \times \Div_{mH} X 
    \]
    consisting of tuples $(p, q, [D], [E]) \in \mathbf{H}(R)$ such that $D - E = \divisor(p/q)$.
\end{lemma}
\begin{proof}
    For brevity, let $d = P_X(t-m)$. Considering $\mathbf{H}$ as a subscheme of the ambient space
    \[
        \mathbf{G} \coloneqq \mathbb{P}((S_X)_t) \times \mathbb{P}((S_X)_t) \times \Gr(d, (S_X)_t) \times \Gr(d, (S_X)_t),
    \]
    its defining equations are computable by \Cref{thm:compute_div}.

    Let $F = (f_i)_{i \in [d]}$ and $G = (g_j)_{j \in [d]}$ be the Stiefel matrices representing $[D]$ and $[E]$, respectively. The condition $D - E = \divisor(p/q)$ is equivalent to 
    \[ q \cdot (I_{D/X})_t = p \cdot (I_{E/X})_t.\] 
    This holds if and only if the $\dim (S_X)_{2t} \times d$ matrices $\Psi \coloneqq (qf_i)_{i \in [d]}$ and $\Omega \coloneqq (pg_j)_{j \in [d]}$ share the same column space. Therefore, by applying \Cref{lem:submodule_criterion} to the block matrix
    \[ (\Psi \mid \Omega), \]
    we obtain the defining equations of $\mathbf{W}$ inside $\mathbf{H}$.

    Since the column spaces of $F$ and $G$ are invariant under the $\GL_d$-action and the projective coordinates are invariant under scaling, the column spaces of $\Psi$ and $\Omega$ are invariant under the action of $\GL_1 \times \GL_1 \times \GL_d \times \GL_d$. Consequently, the defining ideal is invariant under this group action and can be expressed in Pl\"ucker coordinates by \Cref{thm:ideal_conversion}.
\end{proof}

\begin{lemma}\label{lem:computing_L}
    There exists an algorithm to compute the defining homogeneous equations in Pl\"ucker coordinates for the fiber product
    \[
        \mathbf{L} \coloneqq \Div_{mH} X \times_{\Pic_{mH} X} \Div_{mH} X
    \]
    as a closed subscheme of $\Div_{mH} X \times \Div_{mH} X$.
\end{lemma}
\begin{proof}
  We adopt the notation from \Cref{lem:linear_equivalence_pq}. By \Cref{lem:linear_equiv}, the fiber product $\mathbf{L}$ coincides with the image of the subscheme $\mathbf{W}$ under the projection $\mathbf{H} \to \Div_{mH} X \times \Div_{mH} X$ onto the last two factors. Consequently, the defining equations of $\mathbf{L}$ can be computed by standard elimination theory with a Gröbner basis.
\end{proof}

We now proceed to construct the quotient of $\Div_{mH} X$ by the relation $\mathbf{L}$. This quotient is obtained by computing the image of the morphism $\phi \colon \Div_{mH} X \to \Hilb(\Div_{mH} X)$. To clarify the problem, we temporarily consider a more general setting for the following lemma.

Let $B \hookrightarrow \mathbb{P}^\mu$ and $Y \subset \mathbb{P}^\nu \times B$ be explicitly defined closed subschemes. Assume that the projection $\pi_1 \colon Y \to B$ onto the last factor is flat and that $P$ is the Hilbert polynomial of all fibers over closed points. This setup defines a morphism given by
\begin{align*}
\psi \colon B(R) &\to (\Hilb_P \mathbb{P}^\nu)(R) \\
b &\mapsto [\pi_0(\pi_1^{-1}(b))],
\end{align*}
where $\pi_0\colon Y \to \mathbb{P}^\nu$ is the projection onto the 0-th component. Let
\[ Q(s) = \binom{s+\nu}{\nu} - P(s). \]
Take an integer $u \ge \varphi(Q)$. We consider the embedding $\Hilb_P \mathbb{P}^\nu \hookrightarrow \Gr(Q(u), (S_{\mathbb{P}^\nu})_u)$.

\begin{lemma}\label{lem:compute_hilbert_morphism_psi}
    There exists an algorithm to compute the graph $\Gamma_\psi$ of the morphism
    \[\psi \colon B \to \Hilb_P \mathbb{P}^\nu\]
    as a closed subscheme of the ambient space $\mathbb{P}^\mu \times \Gr(Q(u), (S_{\mathbb{P}^\nu})_u)$.
\end{lemma}
\begin{proof}
  Since $u \ge \varphi(Q)$, we may assume that $Y$ is defined by bihomogeneous polynomials $f_0,\dots,f_{n-1}$, each of which has degree $u$ with respect to the variables of $\mathbb{P}^\nu$. For a point $b \in B$, let $f_i(\cdot, b) \in (S_{\mathbb{P}^\nu})_u$ denote the polynomial obtained by specializing the variables of $B$ to the coordinates of $b$.

  Consider points $b \in B(R)$ and $[Z] \in (\Hilb_P \mathbb{P}^\nu)(R)$. Then $\psi(b) = [Z]$ if and only if
  \[ f_i(\cdot, b) \in (I_Z)_u \]
  for all $i \in [n]$. Let $F(b)$ be the $\dim (S_{\mathbb{P}^\nu})_u \times n$ matrix whose columns correspond to $f_i(\cdot, b) \in (S_{\mathbb{P}^\nu})_u$. Let $G$ be a Stiefel matrix representing $[Z]$. Then $\psi(b) = [Z]$ if and only if $\im F(b) \subset \im G$. Therefore, by applying \Cref{lem:submodule_criterion} with $d = Q(u)$ to the block matrix
  \[ (G \mid F(b)), \]
  we obtain the defining equations of $\Gamma_\psi$.

  Clearly, the column space of $G$ is invariant under the $\GL_{Q(u)}$-action. Thus, the ideal given by \Cref{lem:submodule_criterion} is invariant under the $\GL_{Q(u)}$-action on $G$ and is homogeneous in the coordinates of $b$. Consequently, the defining equations of $\Gamma_\psi$ can be expressed in terms of the coordinates of $B$ and the Pl\"ucker coordinates of $[Z]$ by \Cref{thm:ideal_conversion}.
\end{proof}

Identifying $\Pic_{mH} X$ with the image of the morphism $\phi \colon \Div_{mH} X \to \Hilb(\Div_{mH} X)$ yields the following corollary.

\begin{corollary}\label{cor:Gamma_phi}
    There exists an explicit algorithm to compute the defining equations of the graph of the morphism
    \[
        \Div_{mH} X \to \Pic_{mH} X,
    \]
    which sends $[D] \in (\Div_{mH} X)(R)$ to the line bundle $\mathcal O_{X_R}(D)$.
\end{corollary}
\begin{proof}
    This follows by applying \Cref{lem:compute_hilbert_morphism_psi} to the fiber product $\mathbf{L}$ computed in \Cref{lem:computing_L}. The required Hilbert polynomial is determined by \Cref{lem:hilb_Phi}.
\end{proof}

\PicTau
\begin{proof}
    The defining equations of $\Pic_{mH} X$ can be computed from \Cref{cor:Gamma_phi} using standard elimination theory with Gr\"obner bases. The result follows from the isomorphism \[\Pic^\tau X \cong \Pic_{mH} X.\qedhere\]
\end{proof}

\begin{remark}\label{rem:Pic_dependence}
  The closed subscheme structure of $\Pic^\tau X$ computed in \Cref{thm:Pic_tau} depends on the parameters $m$ and $t$ satisfying condition \eqref{cond:big_numbers}. Different parameters yield isomorphic schemes defined by distinct homogeneous equations in distinct ambient spaces. To compute the group law in the next section, we utilize the models associated with $(m,t)$ and $(2m, 2t)$, simply denoted by $\Pic_{mH} X$ and $\Pic_{2mH} X$ respectively. We identify $\Pic^\tau X$ with $\Pic_{mH} X$.
\end{remark}

\subsection{Computing the Group Structure}

The aim of this subsection is to give an explicit algorithm for computing the commutative group scheme structure on $\Pic^\tau X$. More precisely, we compute the morphisms
\begin{align*}
    \alpha   &\colon \Pic^\tau X \times \Pic^\tau X \to \Pic^\tau X, \\
    \iota    &\colon \Pic^\tau X \to \Pic^\tau X, \\
    \epsilon &\colon \Spec k \to \Pic^\tau X,
\end{align*}
which define the addition, the additive inverse, and the identity section. We begin by computing the addition morphism on the moduli space $\Div_{mH} X$. Recall that we identify $\Div_{mH} X$ and $\Div_{2mH} X$ with closed subschemes of Grassmannians via the embeddings
\begin{align*}
    \Div_{mH} X &\hookrightarrow \Gr\big(P_X(t-m), (S_X)_t\big), \\
    \Div_{2mH} X &\hookrightarrow \Gr\big(P_X(2t-2m), (S_X)_{2t}\big).
\end{align*}

\begin{lemma}\label{lem:addition_div}
    There exists an algorithm to compute the addition morphism
    \[
        \sigma \colon \Div_{mH} X \times \Div_{mH} X \to \Div_{2mH} X,
    \]
    which on $R$-points sends a pair $([D],[E])$ to $[D+E]$.
\end{lemma}
\begin{proof}
    For brevity, let $d_M \coloneqq P_X(t-m)$ and $d_L \coloneqq P_X(2t-2m)$. Consider modules $N, M \in \Div_{mH} X(R)$ and $L \in \Div_{2mH} X(R)$. Let $F = (f_i)_{i \in [d_M]}$, $G = (g_j)_{j \in [d_M]}$, and $H = (h_l)_{l \in [d_L]}$ be the Stiefel matrices representing $N$, $M$, and $L$, respectively. We view the columns $f_i$, $g_j$, and $h_l$ as homogeneous polynomials in $S_X$ as well.
    
    Let $\Omega$ be the $\dim(S_X)_{2t} \times d_M^2$ matrix given by 
    \[
        \Omega \coloneqq \big( f_i g_j \big)_{(i, j) \in [d_M]^2}.
    \]
    The condition $\sigma(N, M) = L$, or equivalently $(N, M, L) \in \Gamma_\sigma(R)$, holds if and only if the submodule generated by the products $N \cdot M$ is contained in $L$. In terms of the matrices, this is equivalent to the condition $\im \Omega \subseteq \im H$. Therefore, by applying \Cref{lem:submodule_criterion} with $d = d_L$ to the block matrix
    \[ (H \mid \Omega), \]
    we obtain the defining equations of $\Gamma_\sigma$.
    
    The column spaces of $H$ and $\Omega$ are invariant under the action of $\GL_{d_M} \times \GL_{d_M} \times \GL_{d_L}$. Consequently, the defining ideal obtained from \Cref{lem:submodule_criterion} is invariant under this action and can be expressed in Pl\"ucker coordinates by \Cref{thm:ideal_conversion}.
\end{proof}

\begin{lemma}\label{lem:addition_pic_mH}
    There exists an algorithm to compute the graph $\Gamma_\beta$ of the addition morphism
    \[
        \beta \colon \Pic_{mH} X \times \Pic_{mH} X \to \Pic_{2mH} X,
    \]
    which corresponds to the tensor product of line bundles.
\end{lemma}
\begin{proof}
    Let $\Gamma_\sigma \subset \Div_{mH} X \times \Div_{mH} X \times \Div_{2mH} X$ be the graph of the addition morphism for divisors computed in \Cref{lem:addition_div}.
    Using \Cref{cor:Gamma_phi}, we can compute the quotient morphism
    \[
        \Pi \colon \Div_{mH} X \times \Div_{mH} X \times \Div_{2mH} X \to \Pic_{mH} X \times \Pic_{mH} X \times \Pic_{2mH} X.
    \]
    The graph $\Gamma_\beta$ is precisely the scheme-theoretic image $\Pi(\Gamma_\sigma)$. Its defining equations can be computed using standard elimination theory.
\end{proof}

\begin{lemma}\label{lem:twist_pic}
    There exists an algorithm to compute the graph $\Gamma_\tau$ of the isomorphism
    \[
        \tau \colon \Pic_{mH} X \xrightarrow{\sim} \Pic_{2mH} X,
    \]
    which corresponds to the twisting $\mathcal{L} \mapsto \mathcal{L}(m)$.
\end{lemma}
\begin{proof}
    Let $\Gamma_\sigma \subset \Div_{mH} X \times \Div_{mH} X \times \Div_{2mH} X$ be the graph of the addition morphism computed in \Cref{lem:addition_div}.
    Computing the fiber of the projection to the middle factor $\Gamma_\sigma \to \Div_{mH} X$ over $[mH]$ yields the graph $\Gamma_\rho$ of the morphism
    \[
        \rho \colon \Div_{mH} X \to \Div_{2mH} X
    \]
    defined by $D \mapsto D + mH$.
    As in the proof of \Cref{lem:addition_pic_mH}, the image of $\Gamma_\rho $ under the quotient 
    \[ \Div_{mH} X \times \Div_{2mH} X \to \Pic_{mH} X \times \Pic_{2mH} X\] 
    is the graph $\Gamma_\tau$.
\end{proof}

\begin{lemma}\label{lem:pic_addition}
    There exists an algorithm to compute the graph $\Gamma_\alpha$ of the addition morphism
    \[
        \alpha \colon \Pic^\tau X \times \Pic^\tau X \to \Pic^\tau X.
    \]
\end{lemma}
\begin{proof}
    As mentioned in \Cref{rem:Pic_dependence}, we identify $\Pic^\tau X$ with $\Pic_{mH} X$.
    Then $\alpha$ is the composition of $\beta$ in \Cref{lem:addition_pic_mH} and $\tau^{-1}$ where $\tau$ is the isomorphism in \Cref{lem:twist_pic}.
    In terms of graphs, $\Gamma_{\tau^{-1}}$ is obtained by swapping the two factors of $\Gamma_\tau$.
    The graph $\Gamma_\alpha$ is the image of the intersection
    \[
        (\Gamma_\beta \times \Pic_{mH} X) \cap (\Pic_{mH} X \times \Pic_{mH} X \times \Gamma_{\tau^{-1}}) \subset \Pic_{mH} X \times \Pic_{mH} X\times \Pic_{2mH} X\times \Pic_{mH} X
    \]
    under the projection that eliminates the factor $\Pic_{2mH} X$.
\end{proof}

We are now ready to establish the second main result of this paper.

\PicTauGroup
\begin{proof}
    We identify $\Pic^\tau X$ with $\Pic_{mH} X$. The addition morphism $\alpha$ is computed in \Cref{lem:pic_addition}. The graph of the inverse $\iota$ is obtained as the fiber over $[2mH]$ of the projection of
    \[
        \Gamma_\beta \subset \Pic_{mH} X \times \Pic_{mH} X \times \Pic_{2mH} X
    \]
    to the last factor, where $\beta$ is the morphism in \Cref{lem:addition_pic_mH}. The identity section $\epsilon$ is simply given by the point $[mH]$.
\end{proof}

\section{Applications}\label{sec:applications}

In this section, we present applications of the two main theorems established in this paper. Before proceeding to the algorithms, we briefly discuss the homological interpretation of $\Pic^\tau X$. As mentioned in the introduction, $\Pic^\tau X$ should be conceptually regarded as the dual of the universal integral first homology group of $X$. In fact, over the complex numbers, $\Pic^\tau X$ is naturally isomorphic to the Pontryagin dual of the singular first homology group. Consequently, it encapsulates the full information of $H_1(X, \mathbb{Z})$.

\begin{proposition}\label{prop:pic_tau_C}
    If $k = \mathbb{C}$, there is a natural isomorphism of locally compact groups
    \[
        \Pic^\tau X \cong H_1(X, \mathbb{Z})^\dual,
    \]
    where $(\cdot)^\dual$ denotes the Pontryagin dual.
\end{proposition}
\begin{proof}
    Since $k = \mathbb{C}$, we regard $X$ as a complex manifold equipped with the analytic topology. Consider the morphism between the following exact sequences.
    \begin{equation*}
        \begin{tikzcd}
            0 \arrow[r] & \mathbb{Z} \arrow[r] \arrow[d, equal] & \mathbb{R} \arrow[r] \arrow[d, "2\pi i"] & \mathbb{R}/\mathbb{Z} \arrow[r] \arrow[d, "\exp(2\pi i \cdot)"] & 0 \\
            0 \arrow[r] & \mathbb{Z} \arrow[r, "2\pi i"] & \mathcal{O}_X \arrow[r, "\exp"] & \mathcal{O}_X^\times \arrow[r] & 0
        \end{tikzcd}
    \end{equation*}
    The associated long exact sequences induce the following morphism of exact sequences.
    \begin{equation*}
        \begin{tikzcd}
            H^1(X, \mathbb{Z}) \arrow[r] \arrow[d, equal] 
            & H^1(X, \mathbb{R}) \arrow[r] \arrow[d, "\sim", sloped] 
            & H^1(X, \mathbb{R}/\mathbb{Z}) \arrow[r] \arrow[d] 
            & H^2(X, \mathbb{Z})_{\tor} \arrow[r] \arrow[d, equal] 
            & 0 \\
            H^1(X, \mathbb{Z}) \arrow[r] 
            & H^1(X, \mathcal{O}_X) \arrow[r] 
            & \Pic^\tau X \arrow[r] 
            & H^2(X, \mathbb{Z})_{\tor} \arrow[r] 
            & 0
        \end{tikzcd}
    \end{equation*}
    The vertical map $H^1(X, \mathbb{R}) \to H^1(X, \mathcal{O}_X)$ is an isomorphism by Hodge theory. Consequently, the Five Lemma and the Universal Coefficient Theorem imply
    \[
        \Pic^\tau X \cong H^1(X, \mathbb{R}/\mathbb{Z}) \cong \Hom(H_1(X, \mathbb{Z}), \mathbb{R}/\mathbb{Z}). \qedhere 
    \]
\end{proof}

Although Proposition \ref{prop:pic_tau_C} has no direct analogue over general fields, $\Pic^\tau X$ remains the conceptual dual of the universal integral first homology. This conceptual framework provides the intuition that $\pi^\et_1(X, x)^{\ab}$ and $H_{\et}^1(X, \mathbb{Z}/n\mathbb{Z})$ can be recovered from $\Pic^\tau X$. Crucially, as we will demonstrate, in positive characteristic, the non-reduced structure of $\Pic^\tau X$ encodes the essential $p$-power torsion data of these objects. However, before computing these homological invariants, we first present algorithms to compute the requisite geometric objects.

\subsection{Computing \texorpdfstring{$\Alb X$ and $(\NNS X)_\tor$}{Alb X and (NS X)\_tor}}

As preliminary steps, we compute the Albanese variety and the torsion subgroup of the N\'eron-Severi group scheme. Recall that the Albanese variety $\Alb X$ is the dual abelian variety of the Picard variety $(\Pic^0 X)_\red$.

\begin{proposition}\label{prop:Alb}
    There exists an algorithm to compute the Albanese variety $\Alb X$ of $X$ together with its group scheme structure.
\end{proposition}
\begin{proof}
    First, we compute $\Pic^\tau X$ using \Cref{thm:Pic_tau}. Then, we compute the reduced identity component $(\Pic^0 X)_\red$ by computing the radical of the ideal defining $\Pic^0 X$. We apply \Cref{thm:Pic_tau} once more to compute
    \[
        \Alb X = (\Pic^0 X)_\red^\dual = \Pic^\tau((\Pic^0 X)_\red).
    \]
    We can also compute the group scheme structure of $\Alb X$ by \Cref{thm:group_structure}.
\end{proof}

To compute $\pi^\et_1(X, x)^{\ab}$ and $H_{\et}^1(X, \mathbb{Z}/n\mathbb{Z})$, we also require the notion of the N\'eron-Severi group scheme \cite[p.~70]{Stix}. For our purposes, it suffices to focus on its torsion subgroup.

\begin{definition}\label{def:NS_tor}
    The torsion subgroup of the N\'eron-Severi group scheme of $X$ is defined by the exact sequence
    \[
        0 \to (\Pic^0 X)_\red \to \Pic^\tau X \to (\NNS X)_\tor \to 0.
    \]
\end{definition}

Before describing the algorithms, we note that by \cite[Theorem~3.2.1 and Lemma~3.3.7]{Brion},
\[(\NNS X)_\tor \cong \Spec H^0(\Pic^\tau X, \mathcal{O}_{\Pic^\tau X}),\]
which implies that $(\NNS X)_\tor$ is a finite group scheme. Consequently, the computation of $(\NNS X)_\tor$ reduces to determining the Hopf algebra structure of the global sections of $\Pic^\tau X$. We review the standard computation of global sections to fix notation for the Hopf algebra structure.

\begin{lemma}\label{lem:ring_structure}
    Let $Y \subset \mathbb{P}^n$ be a projective scheme over $k$. There exists an algorithm to compute the finite-dimensional $k$-algebra $(S_Y)_0 \coloneqq \Gamma(Y, \mathcal{O}_Y)$ represented by \v{C}ech cocycles. In particular, the algorithm determines a basis and explicitly describes the unit map $e \colon k \to (S_Y)_0$ and the multiplication map $m \colon (S_Y)_0 \otimes_k (S_Y)_0 \to (S_Y)_0$ in terms of this basis.
\end{lemma}
\begin{proof}
    Let $S \coloneqq k[x_0, \dots, x_n]$ be the coordinate ring of $\mathbb{P}^n$. Let $t$ be an integer such that the ideal sheaf $\mathscr{I}_Y$ is $t$-regular, for instance $t \ge \varphi(Y)$. Since $H^1(\mathbb{P}^n, \mathscr{I}_Y(t)) = 0$, we have
    \[ (S_Y)_t \cong S_t / (I_Y)_t. \]
    As $S$ and $I_Y$ are explicitly given, the vector space $(S_Y)_t$ is computable. For any global section $f \in (S_Y)_0$ and any index $i \in \{0, \dots, n\}$, we have $x_i^t f \in (S_Y)_t$. Consequently, on the open set $D(x_i)$, the section $f$ can be represented as $f_i/x_i^t$ for some $f_i \in (S_Y)_t$.

    Therefore, any $f \in (S_Y)_0$ is determined by a collection of local sections $f|_{D(x_i)} = f_i/x_i^t$ with $f_i \in (S_Y)_t$, subject to the compatibility condition $f_i/x_i^t = f_j/x_j^t$ on $D(x_i x_j)$. This condition is equivalent to $x_j^t f_i = x_i^t f_j$ in $(S_Y)_{2t}$. Thus, $(S_Y)_0$ is isomorphic to the kernel of the map
    \begin{align*}
        \psi \colon \prod_{0 \le i \le n} (S_Y)_t &\to \prod_{0 \le i < j \le n} (S_Y)_{2t}, \\
        (f_i)_i &\mapsto (x_j^t f_i - x_i^t f_j)_{i,j}.
    \end{align*}
    Since this is a linear map between computable finite-dimensional vector spaces, we can compute a basis for $(S_Y)_0$.

    Moreover, the unit map $e$ sends $1_k$ to the tuple $(x_0^t, \dots, x_n^t)$. The multiplication map $m$ sends $(f_i)_i$ and $(g_i)_i$ to the unique $(h_i)_i$ such that $x_i^t h_i = f_i g_i$ in $(S_Y)_{2t}$ for all $i$. By computing the multiplication for the basis elements of $(S_Y)_0 \otimes (S_Y)_0$, we can explicitly determine the full map $m \colon (S_Y)_0 \otimes (S_Y)_0 \to (S_Y)_0$.
\end{proof}

\begin{lemma}\label{lem:hopf_structure}
    Suppose further that $Y \subset \mathbb{P}^n$ is a projective group scheme. There exists an algorithm to compute the finite-dimensional Hopf algebra structure on $(S_Y)_0 \coloneqq \Gamma(Y, \mathcal{O}_Y)$. In particular, the algorithm explicitly describes the counit $\epsilon \colon (S_Y)_0 \to k$, the antipode $\imath \colon (S_Y)_0 \to (S_Y)_0$, and the comultiplication $\mu \colon (S_Y)_0 \to (S_Y)_0 \otimes_k (S_Y)_0$ with respect to the basis computed in \Cref{lem:ring_structure}.
\end{lemma}
\begin{proof}
    We retain the notation from \Cref{lem:ring_structure} and let $(f_i)_{i \in [n+1]}$ be the tuple representing a global section $f \in (S_Y)_0$. Fix an index $i$ such that the coordinate $x_i$ does not vanish at the identity point of $Y$. The counit $\epsilon(f)$ is computed by evaluating the local section $f_i/x_i^t$ at the identity point.

    Next, we compute the antipode $g = \imath(f)$, represented by a tuple $(g_i)_{i \in [n+1]}$. Let $\Gamma \subset Y \times Y$ be the graph of the morphism $Y \to Y$ representing the inverse operation of the group scheme $Y$. We assume that $t$ is chosen such that $\Gamma$ is $t$-regular with respect to $\mathcal{O}_{\Gamma}(1,1)$; such a $t$ can be computed via the Gotzmann number of the image of $\Gamma$ under the Segre embedding. For clarity, we employ coordinates $x = (x_0, \dots, x_n)$ and $y = (y_0, \dots, y_n)$ for the two factors of the ambient space $\mathbb{P}^n \times \mathbb{P}^n$ of $\Gamma$. Then the condition $g = \imath(f)$ is equivalent to
    \[
        \frac{g_i(x)}{x_i^t} = \frac{f_j(y)}{y_j^t} \text{ for all } i, j \in [n+1].
    \]
    This condition translates to the equation
    \[
        y_j^t g_i(x) = x_i^t f_j(y) \text{ in } (S_\Gamma)_{(t,t)}.
    \]
    We can now use elementary linear algebra to solve for the unique coefficients of the tuple $(g_i)_{i \in [n+1]}$ satisfying the above equations.
    
    The comultiplication is computed analogously by considering the graph of the multiplication morphism in $Y \times Y \times Y$.
\end{proof}

\begin{corollary}\label{cor:NS_tor}
    There exists an algorithm to compute the torsion subgroup scheme $(\NNS X)_\tor$ of the N\'eron-Severi group scheme.
\end{corollary}
\begin{proof}
    This follows from \Cref{lem:hopf_structure} together with the isomorphism
    \[(\NNS X)_\tor \cong \Spec H^0(\Pic^\tau X, \mathcal{O}_{\Pic^\tau X}). \qedhere \]
\end{proof}

\begin{remark}
    Just as $\Pic^\tau X$ corresponds to the dual of the first homology, the Albanese variety $\Alb X$ corresponds to the dual of the first cohomology. Furthermore, regarding the torsion parts, let $\pi^N_1(X, x)$ denote Nori's fundamental group scheme \cite{Nori} for a rational point $x \in X(k)$. By \cite[Theorem~6.4]{Kwe2}, we have the Cartier duality
    \[ \pi^N_1(X, x)^\ab_\tor \cong (\NNS X)_\tor^\dual.\]
    This isomorphism reflects the topological correspondence wherein $\pi^N_1(X, x)^\ab_\tor$ and $(\NNS X)_\tor$ correspond to the torsion parts of the first homology and the second cohomology, respectively.
\end{remark}

Before proceeding to the next subsection, we establish the following lemma. Here, a finite group scheme is described by a finite-dimensional Hopf algebra whose structure is explicitly determined in terms of a basis.

\begin{lemma}\label{hyp:group_at_bar_k}
    Given an explicit finite group scheme $G$ over $k$, there exists an algorithm to compute the $\Aut(\bar{k}/k)$-module structure on $G(\bar{k})$. To be more precise, the algorithm determines
    \begin{enumerate}
        \item the minimal finite extension $L$ of $k$ satisfying $G(\bar{k}) = G(L)$,
        \item the left group action $\Aut(L/k) \times G(L) \to G(L)$, and
        \item the abstract group structure of $G(L)$.
    \end{enumerate}
\end{lemma}
\begin{proof}
    Let $G = \Spec A$ with a basis $\{x_0, \dots, x_{n-1}\}$ of $A$. Expressing the products $x_i x_j$ in terms of this basis, we present $A$ as a quotient of $k[x_0, \dots, x_{n-1}]$, which induces a closed embedding $G \hookrightarrow \mathbb{A}^n$. Let $I_G$ be the ideal defining $G$.
    
    First, we determine the minimal field $L$ and the set $G(L)$. We compute a Gr\"obner basis of $I_G$ with respect to the lexicographic order to find the common zeros. Since $I_G$ is zero-dimensional, under a lex order with $x_{n-1} \succ \cdots \succ x_0$, a Gr\"obner basis contains a univariate polynomial in $x_0$. We find all its roots; if the polynomial does not split into linear factors, we extend the base field to contain these roots. Substituting each root into $x_0$, we repeat the process of computing the Gr\"obner basis and solving for the subsequent variables. Recursively applying this procedure yields the splitting field $L$ and the set of solutions $G(L) \subset \mathbb{A}^n(L)$.
    
    To compute $\Aut(L/k)$, we consider a basis of $L$ over $k$. We compute the minimal polynomials for these basis elements and consider all assignments mapping the generators to their conjugates. There are finitely many such assignments; we collect those that induce well-defined $k$-algebra morphisms, which form the group $\Aut(L/k)$. The action $\Aut(L/k) \times G(L) \to G(L)$ is then computed by applying each automorphism component-wise to the points in $G(L)$.
    
    Finally, we compute the group structure on $G(L)$. The identity element is obtained directly from the section $\epsilon \colon \Spec k \to G$ expressed in coordinates of $\mathbb{A}^n(k)$. From the comultiplication $\mu \colon A \to A \otimes A$, we derive the equations defining the graph $\Gamma \subset \mathbb{A}^n \times \mathbb{A}^n \times \mathbb{A}^n$ of the multiplication morphism. Applying the same Gr\"obner basis method as above, we compute the set of points $\Gamma(L) \subset G(L) \times G(L) \times G(L)$, which provides the multiplication table for $G(L)$.
\end{proof}

\subsection{Computing Homological Objects and Cohomology}

We are now ready to compute the homological objects. In the remainder of this section, let $x$ be a geometric point of $X$, and let $\ell$ be a prime number distinct from $\charac k$. If the characteristic of $k$ is positive, we denote it by $p$; otherwise, we disregard everything related to $p$.

Under these hypotheses, we first address the computation of the \'etale fundamental group. While the full \'etale fundamental group $\pi^\et_1(X, x)$ is generally highly complex, its abelianization is tractable.

\begin{proposition}\label{prop:fund_group_structure}
    Let $g$ be the dimension of $\Alb X$. If $\charac k = p > 0$, let $\sigma$ denote the $p$-rank of $\Alb X$. Then
    \[
        \pi^\et_1(X_{\bar{k}}, x)^{\ab} \cong 
        \begin{cases}
            \widehat{\mathbb{Z}}^{2g} \times (\NNS X)_\tor^\dual(\bar{k}) & \text{if } \charac k = 0, \\[1em]
            \displaystyle \Bigg( \prod_{\substack{\ell \neq p \\ \text{prime}}} \mathbb{Z}_\ell^{2g} \Bigg) \times \mathbb{Z}_p^\sigma \times (\NNS X)_\tor^\dual(\bar{k}) & \text{if } \charac k = p > 0.
        \end{cases}
    \]
\end{proposition}
\begin{proof}
    According to \cite[Proposition~69]{Stix}, there exists an exact sequence
    \[
        0 \to (\NNS X)_\tor^\dual(\bar{k}) \to \pi^\et_1(X_{\bar{k}}, x)^{\ab} \to \pi^\et_1(\Alb X_{\bar{k}}, 0) \to 0.
    \]
    Because $\Alb X_{\bar{k}}$ is an abelian variety, 
    \[\pi^\et_1(\Alb X_{\bar{k}}, 0) \cong \varprojlim_{n} (\Alb X_{\bar{k}})[n](\bar{k}).\]
    Since this Tate module is a projective profinite abelian group, the exact sequence splits, and the result follows.
\end{proof}

Let $Y$ be a projective group scheme. The multiplication-by-$n$ morphism $[n] \colon Y \to Y$ can be explicitly computed by iteratively composing the diagonal morphism $\Delta \colon Y \to Y \times Y$ and the multiplication morphism $\mu \colon Y \times Y \to Y$. Consequently, the $n$-torsion subgroup scheme $Y[n]$ is computable as the scheme-theoretic fiber of the identity section $\epsilon \colon \Spec k \to Y$ under the morphism $[n]$.

\firstHomology
\begin{proof}
    We compute $\Alb X$ via \Cref{prop:Alb} and determine its dimension $g$ from its Hilbert polynomial. We then compute the kernel $(\Alb X)[p]$ of the multiplication-by-$p$ map, apply \Cref{hyp:group_at_bar_k} to find the finite group $(\Alb X)[p](\bar{k})$, and determine its order $p^\sigma$. Subsequently, we compute $(\NNS X)_\tor$ using \Cref{cor:NS_tor}, construct its Cartier dual $(\NNS X)_\tor^\dual$ by taking the dual vector space of the coordinate ring, and apply \Cref{hyp:group_at_bar_k} to obtain $(\NNS X)_\tor^\dual(\bar{k})$. The result now follows from the isomorphism in \Cref{prop:fund_group_structure}.
\end{proof}

Although the structure of $\pi^\et_1(X, x)$ is determined by \Cref{thm:first_homology}, it is essentially an infinite object. Consequently, providing a finite description of the continuous $\Aut(\bar{k}/k)$-action is generally infeasible. In what follows, we focus on computing finite objects equipped with explicit Galois actions. Thus, we introduce the following results.

\begin{proposition}\label{prop:ab_fund_group_dual}
    There is a canonical isomorphism of $\Aut(\bar{k}/k)$-modules
    \[
        \frac{\pi^\et_1(X_{\bar{k}}, x)^{\ab}}{n\pi^\et_1(X_{\bar{k}}, x)^{\ab}} \cong (\Pic^\tau X_{\bar{k}})[n]^\dual(\bar{k}),
    \]
    where $(\cdot)^\dual$ denotes the Cartier dual.
\end{proposition}
\begin{proof}
    By \cite[Proposition~3.4]{Antei},\footnote{Although Antei assumes at the beginning of Subsection~3.1 that $\Pic^0 X$ is an abelian scheme, the proof of this result does not rely on its smoothness.} we have a canonical isomorphism of affine group schemes
    \[
       \pi^N_1(X_{\bar{k}}, x)^{\ab} \cong \varprojlim_n (\Pic^\tau X_{\bar{k}})[n]^\dual,
    \]
    which is compatible with the $\Aut(\bar{k}/k)$-action. Consequently, we obtain an isomorphism of finite group schemes
     \[ \frac{\pi^N_1(X_{\bar{k}}, x)^{\ab}}{n\pi^N_1(X_{\bar{k}}, x)^{\ab}} \cong (\Pic^\tau X_{\bar{k}})[n]^\dual.\]
    The result follows from applying the exact functor $G \mapsto G(\bar{k})$.
\end{proof}

For two group schemes $G$ and $H$ over $k$, we denote by $\Hom_{k\text{-}\mathrm{Grp}}(G, H)$ the group of homomorphisms of $k$-group schemes.

\begin{proposition}\label{prop:fppf_duality}
    Let $G$ be a finite commutative group scheme over $k$. Then there is a canonical isomorphism of $\Aut(\bar{k}/k)$-modules
    \[
        H^1_{\mathrm{fppf}}(X, G) \cong \Hom_{k\text{-}\mathrm{Grp}}(G^\dual, \Pic_{X/k}),
    \]
    where $(\cdot)^\dual$ denotes the Cartier dual.
\end{proposition}
\begin{proof}
    See \cite[Exposé~XI, Remarques~6.11, p.~309]{SGA1}.
\end{proof}

For two groups $A$ and $B$, we denote by $\Hom_{\mathrm{Grp}}(A, B)$ the group of homomorphisms.

\begin{lemma}\label{lem:first_cohomology_dual}
    We have a canonical isomorphism of $\Aut(\bar{k}/k)$-modules
    \[ 
        H_{\et}^1(X_{\bar{k}}, \mathbb{Z}/n\mathbb{Z})^\dual \cong \frac{\pi^\et_1(X_{\bar{k}}, x)^{\ab}}{n\pi^\et_1(X_{\bar{k}}, x)^{\ab}}, 
    \]
    where $(\cdot)^\dual$ denotes the Pontryagin dual of finite abelian groups, defined by $\Hom_{\mathrm{Grp}}(\cdot, \mathbb{Q}/\mathbb{Z})$.
\end{lemma}
\begin{proof}
    Let $\underline{\mathbb{Z}/n\mathbb{Z}}$ be the constant group scheme over $\bar{k}$ associated to $\mathbb{Z}/n\mathbb{Z}$. By \Cref{prop:fppf_duality},
    \[
        H_{\et}^1(X_{\bar{k}}, \mathbb{Z}/n\mathbb{Z}) \cong \Hom_{\bar{k}\text{-}\mathrm{Grp}}(\mmu_n, \Pic X_{\bar{k}}).
    \]
    Since $\mmu_n$ is annihilated by $n$, any homomorphism from $\mmu_n$ to $\Pic X_{\bar{k}}$ factors through $(\Pic^\tau X_{\bar{k}})[n]$. Applying Cartier duality, we obtain
    \begin{align*}
        H_{\et}^1(X_{\bar{k}}, \mathbb{Z}/n\mathbb{Z}) &\cong \Hom_{\bar{k}\text{-}\mathrm{Grp}}(\mmu_n, (\Pic^\tau X_{\bar{k}})[n]) \\
        &\cong \Hom_{\bar{k}\text{-}\mathrm{Grp}}((\Pic^\tau X_{\bar{k}})[n]^\dual, \underline{\mathbb{Z}/n\mathbb{Z}}).
    \end{align*}
    Since $\underline{\mathbb{Z}/n\mathbb{Z}}$ is a constant group scheme, taking $\bar{k}$-valued points yields an isomorphism of abstract groups. Using \Cref{prop:ab_fund_group_dual}, we have
    \begin{align*}
        H_{\et}^1(X_{\bar{k}}, \mathbb{Z}/n\mathbb{Z}) 
        &\cong \Hom_{\mathrm{Grp}}((\Pic^\tau X_{\bar{k}})[n]^\dual(\bar{k}), \mathbb{Z}/n\mathbb{Z})\\
        &\cong \Hom_{\mathrm{Grp}}\!\left( \frac{\pi^\et_1(X_{\bar{k}}, x)^{\ab}}{n\pi^\et_1(X_{\bar{k}}, x)^{\ab}}, \mathbb{Z}/n\mathbb{Z} \right) \\
        &\cong \Hom_{\mathrm{Grp}}\!\left( \frac{\pi^\et_1(X_{\bar{k}}, x)^{\ab}}{n\pi^\et_1(X_{\bar{k}}, x)^{\ab}}, \mathbb{Q}/\mathbb{Z} \right).
    \end{align*}
    All isomorphisms considered here are canonical isomorphisms of $\Aut(\bar{k}/k)$-modules.
\end{proof}

\begin{proposition}\label{prop:compute_fund_group_mod_n}
    For any integer $n > 0$, there exists an algorithm to compute the $\Aut(\bar{k}/ k)$-module 
    \[ 
        \frac{\pi^\et_1(X_{\bar{k}}, x)^{\ab}}{n\pi^\et_1(X_{\bar{k}}, x)^{\ab}}.
    \]
\end{proposition}
\begin{proof}
    The algorithm proceeds as follows. First, we compute $\Pic^\tau X$ and its group structure using \Cref{thm:Pic_tau} and \Cref{thm:group_structure}. Second, we compute the coordinate ring of $(\Pic^\tau X)[n]$ via \Cref{lem:hopf_structure}. Third, we compute its Cartier dual $(\Pic^\tau X)[n]^\dual$ by taking the dual vector space of the coordinate ring. Fourth, we apply \Cref{hyp:group_at_bar_k} to compute the $\Aut(\bar{k}/k)$-module
    \[
        (\Pic^\tau X_{\bar{k}})[n]^\dual(\bar{k}),
    \]
    which yields the desired result by \Cref{prop:ab_fund_group_dual}.
\end{proof}

We now turn to the computation of \'etale cohomology groups. Previously, Madore and Orgogozo \cite[0.6]{MadoreOrgogozo} provided algorithms to compute $H^i(X, \mathbb{Z}/\ell^n \mathbb{Z})$ for all $i$. However, analogous results for $H^i(X, \mathbb{Z}/p^n \mathbb{Z})$ where $p = \charac k > 0$ have not been established. We present an algorithm to compute this for the case $i = 1$.

\firstCohomology
\begin{proof}
    By \Cref{lem:first_cohomology_dual}, it suffices to compute the dual of the finite abelian group obtained in \Cref{prop:compute_fund_group_mod_n}.
\end{proof}

\begin{remark}
    These examples highlight the significance of our results. For both $\pi^\et_1(X_{\bar{k}}, x)^{\ab}$ and $H_{\et}^1(X_{\bar{k}}, \mathbb{Z}/n\mathbb{Z})$, the $p$-power torsion subgroup of the Néron-Severi group has no impact on the results. In contrast, the non-reduced structure of the Picard scheme influences the outcome. Consequently, computing the Picard group alone is insufficient for these applications. It is essential to determine both the full scheme structure and the group structure of the Picard scheme.
\end{remark}

The $p$-power torsion of the N\'eron-Severi group is related to the following result.

\begin{proposition}
    For any integer $n > 0$, there exists an algorithm to compute the $\Aut(\bar{k}/ k)$-module $H^1_{\mathrm{fppf}}(X_{\bar{k}}, \mmu_n)$.
\end{proposition}
\begin{proof}
    By \Cref{prop:fppf_duality}, we have
    \begin{align*}
        H^1_{\mathrm{fppf}}(X_{\bar{k}}, \mmu_n) &\cong \Hom_{\bar{k}\text{-}\mathrm{Grp}}(\underline{\mathbb{Z}/n\mathbb{Z}}, \Pic X_{\bar{k}}) \\
        &\cong (\Pic X)[n](\bar{k}).
    \end{align*}
    Therefore, the result is obtained by following the algorithm in \Cref{prop:compute_fund_group_mod_n}, but omitting the steps of taking the Cartier dual.
\end{proof}

\begin{remark}
    The results presented in this section can be applied to other fundamental groups or cohomology theories as well. However, integral crystalline cohomology is an exception due to the pathological nature of its torsion.

    If $X$ is a complex variety, the torsion subgroup $H^2(X, \mathbb{Z})_\tor$ coincides with the torsion subgroup of the N\'eron-Severi group $\NS X$. Analogously, in the context of crystalline cohomology, one might naturally conjecture that the torsion part of the second crystalline cohomology $H^2_{\mathrm{crys}}(X/W)_\tor$ corresponds to the covariant Dieudonn\'e module $\mathbb{D}((\NNS X)_\tor)$ of $(\NNS X)_\tor$. However, as noted in \cite[p.~600]{Illusie}, this is not true in general. Specifically, while $\mathbb{D}((\NNS X)_\tor)$ injects into $H^2_{\mathrm{crys}}(X/W)_\tor$, the latter is generally strictly larger. The torsion elements that are not explained by the Dieudonn\'e module are referred to as \emph{exotic torsion}. See \cite{Oda} for examples exhibiting exotic torsion and \cite{Josh} for cases where exotic torsion vanishes.
\end{remark}
\bibliographystyle{plain}
\bibliography{mybib}

\end{document}